%% file: main.tex
\documentclass{article}
\input{setup.tex}
\usepackage{authblk}

\title{The Frobenius problem over real number fields}

\author[a]{Alex Feiner\footnote{Corresponding author. Email address: alexander\_feiner@brown.edu. ORCID iD: 0000-0001-9588-7860.}}
\author[b]{Zion Hefty}
\affil[a]{Department of Mathematics, Brown University, 151 Thayer St, Providence, RI 02912.}
\affil[b]{Department of Mathematics and Statistics, Grinnell College, 1115 8th Avenue Grinnell, IA 50112.}

\date{}
\begin{document}
\maketitle

\begin{abstract}
    Given a number field $K$ that is a subfield of the real numbers, we generalize the notion of the classical Frobenius problem to the ring of integers $\Ofrak_K$ of $K$ by describing certain Frobenius semigroups, $\Frob(\al_1,\dots,\al_n)$, for appropriate elements $\al_1,\dots,\al_n\in\Ofrak_K$. We construct a partial ordering on $\Frob(\al_1,\dots,\al_n)$, and show that this set is completely described by the maximal elements with respect to this ordering. We also show that $\Frob(\al_1,\dots,\al_n)$ will always have finitely many such maximal elements, but in general, the number of maximal elements can grow without bound as $n$ is fixed and $\al_1,\dots,\al_n\in\Ofrak_K$ vary. Explicit examples of the Frobenius semigroups are also calculated for certain cases in real quadratic number fields. 
\end{abstract}


\section*{Acknowledgements}
This research was supported by NSF (DMS) [grant number 1950563]. The authors would like to thank Peter Johnson for many helpful conversations had when conducting this research and comments made in editing the paper. The authors would also like to thank Overtoun Jenda for helpful discussions had in the early stages of the paper. 

\section*{Statements and Declarations}
Declarations of interest: none.

\section{Introduction and Summary of Results}
It is well known that if $\al_1,\dots,\al_n\in\Nbb\coloneqq\{0,1,2,\dots\}$ are nonzero and coprime, then there is some smallest positive integer $\chi(\al_1,\dots,\al_n)$ with the property that for any integer $N\geqslant \chi(\al_1,\dots,\al_n)$, there are natural numbers $x_1,\dots,x_n\in\Nbb$ for which
    $$x_1\al_1+\cdots+x_n\al_n=N.$$
The classical Frobenius problem concerns explicitly finding the number $\chi(\al_1,\dots,\al_n)$. When $n=2$, it is known (see \cite{FrobBook}) that 
    $$\chi(\al_1,\al_2)=(\al_1-1)(\al_2-1),$$
and more complicated formulas are known for $\chi$ when $n=3$ (see \cite{TRIPATHI}). We can restate the classical Frobenius problem as follows: Define a submonoid $\SG(\al_1,\dots,\al_n)$ of $\Nbb$ by
    $$\SG(\al_1,\dots,\al_n)=\left\{\left.\sum_{i=1}^{n}x_i\al_i \, \right| x_1,\dots,x_n\in\Nbb\right\}.$$
Then the classical Frobenius problem is to determine the semigroup
    $$\Frob(\al_1,\dots,\al_n)=\{w\in\SG(\al_1,\dots,\al_n) \mid w+\Nbb\subseteq\SG(\al_1,\dots,\al_n)\}=\chi(\al_1,\dots,\al_n)+\Nbb,$$
and the above shows that in the case $n=2$, we have
    $$\Frob(\al_1,\al_2)=(\al_1-1)(\al_2-1)+\Nbb.$$
Using this new statement of the classical Frobenius problem, we can generalize to certain commutative rings with unity in the spirit of Johnson and Looper's paper \cite{FrobInDoms}:

\begin{definition}\label{Frobenius Template Definition}
Let $R$ be a commutative ring with unity that is finitely generated as a $\Zbb$-module. Then we define a \emph{Frobenius template} (or simply \emph{template}) for $R$ to be a triple $\Tcal=(A,C,U)$ consisting of
\begin{enumerate}[label=(\arabic*)]
    \item a subset $A\subseteq R$ containing $1$;
    \item a nonzero additive monoid $C\subseteq R$; and
    \item a function $U$ that assigns to each collection of nonzero distinct elements $\al_1,\dots,\al_n\in A$ that generate $R$ as a $\Zbb$-module, an additive submonoid $U(\al_1,\dots,\al_n)$ of $R$ for which 
        $$U(\al_1,\dots,\al_n)\supseteq\SG_\Tcal(\al_1,\dots,\al_n)\coloneqq\left\{\left.\sum_{i=1}^{n}x_i\al_i \, \right|  x_1,\dots,x_n\in C\right\}.$$
\end{enumerate}
We call the Frobenius template $\Tcal=(A,C,U)$ \emph{Frobenius} if for all nonzero $\al_1,\dots,\al_n\in A$ that generate $R$ as a $\Zbb$-module, there is some $w\in\SG_\Tcal(\al_1,\dots,\al_n)$ for which 
    $$w+U(\al_1,\dots,\al_n)\subseteq\SG_\Tcal(\al_1,\dots,\al_n).$$
\end{definition}

\noindent
This notion of a Frobenius template is slightly different from what originally appeared in \cite{FrobInDoms}, since we have replaced the notion of $\al_1,\dots,\al_n\in A$ being coprime (i.e., having no common non-unit divisors) with the much stronger notion of $\al_1,\dots,\al_n\in A$ generating $R$ as a $\Zbb$-module. Of course, an arbitrary ring $R$ may not be finitely generated as a $\Zbb$-module, so this assumption is added into the above definition in order to avoid having a useless concept in cases where this fails. While the rest of this paper deals with cases where $R$ has characteristic zero and is a finitely generated free $\Zbb$-module, it could be interesting to consider Frobenius templates for rings in which one of these conditions fails, in which case if $R$ is not finitely generated as a $\Zbb$-module, then the definition of the Frobenius template would have to be altered in order to allow for an infinite indexed family of elements $\{\al_i\}_{i\in I}\subseteq A$. Another key property that all rings considered in Frobenius templates in this paper will have is that they are a subset of the real numbers, so they inherit the standard total ordering that $\Rbb$ has. It could also be interesting to consider rings without this property. \\
\\
Based on the above results about rephrasing the classical Frobenius problem in terms of determining certain semigroups, we can generalize the Frobenius problem to some rings:

\begin{definition}\label{Frobenius Problem Definition}
If a template $\Tcal=(A,C,U)$ for a ring $R$ is Frobenius, then the \emph{Frobenius problem} associated to $\Tcal$ is to determine, for each collection of nonzero elements $\al_1,\dots,\al_n\in A$ that generate $R$ as a $\Zbb$-module, the \emph{Frobenius semigroup} 
    $$\Frob_\Tcal(\al_1,\dots,\al_n)\coloneqq\{w\in\SG_\Tcal(\al_1,\dots,\al_n) \ | \ w+U(\al_1,\dots,\al_n)\subseteq\SG_\Tcal(\al_1,\dots,\al_n)\}.$$
We will often drop the subscript $\Tcal$ if the template in use is clear from context. 
\end{definition}

\noindent
Given a Frobenius template $\Tcal=(A,C,U)$ over a ring $R$ and nonzero elements $\al_1,\dots,\al_n\in A$ that generate $R$ as a $\Zbb$-module, the requirement that $C\subseteq R$ is a monoid shows that $\SG(\al_1,\dots,\al_n)$ is also a monoid. It is then an immediate consequence of definition \ref{Frobenius Problem Definition} that the following are equivalent:
    $$0\in\Frob_\Tcal(\al_1,\dots,\al_n) \iff \SG_\Tcal(\al_1,\dots,\al_n)=U(\al_1,\dots,\al_n) \iff \Frob_\Tcal(\al_1,\dots,\al_n)=\SG_\Tcal(\al_1,\dots,\al_n).$$
From definition \ref{Frobenius Problem Definition}, we see that the classical Frobenius problem is about the ring $R=\Zbb$, and the Frobenius template in question is $\Tcal=(\Nbb,\Nbb,\Nbb)$ (where the third $\Nbb$ is the constant function that assigns to any tuple of natural numbers the submonoid $\Nbb$ of $\Zbb$), since nonzero integers $\al_1,\dots,\al_n\in\Nbb$ generate $\Zbb$ as a $\Zbb$-module if and only if they are coprime. Work has been done on finding and studying certain interesting Frobenius templates (with the requirement of elements generating the ring as a $\Zbb$-module replaced with other requirements) for the ring $R=\Zbb\left[\sqrt{m}\right]$, where $m\in\Zbb$ is not a square, in \cite{FrobGaussian}, \cite{2kindsofFrob}, \cite{FrobInDoms}, \cite{FrobPrimes},  \cite{2variableFrob}, and \cite{SimpFrob}. \\
\\
In this paper, we look into creating and studying an interesting Frobenius template for the ring of integers of a number field that is a subfield of the real numbers (henceforth, such number fields will be referred to as \emph{real number fields}). Some results are based on similar results in \cite{TotallyReal}, but we weaken the restriction there of the number field being totally real to that of the number field being a subfield of the real numbers. 

\begin{definition}
Let $K$ be a real number field with ring of integers $\Ofrak_K$, and define $\Ofrak_K^+=\Ofrak_K\cap[0,\infty)$. For any $\al_1,\dots,\al_n\in\Ofrak_K^+$ that generate $\Ofrak_K$ as a $\Zbb$-module, let the \emph{positive rational cone} generated by $\al_1,\dots,\al_n$ be the set
    $$C_\Qbb(\al_1,\dots,\al_n)=\left\{\left.\sum_{i=1}^{n}x_i\al_i \, \right|  x_1,\dots,x_n\in \Qbb_{\geqslant0}\right\}\subseteq K\cap[0,\infty).$$
Let $C_\Qbb\cap\Ofrak_K$ denote the function that assigns to each such collection $\al_1,\dots,\al_n$ the submonoid $C_\Qbb(\al_1,\dots,\al_n)\cap\Ofrak_K$ of $\Ofrak_K^+$. 
\end{definition}

\noindent
With the above notation, we have that $C_\Qbb(\al_1,\dots,\al_n)\cap\Ofrak_K\supseteq\Nbb\al_1+\cdots+\Nbb\al_n$, so we can form the Frobenius template 
    $$\Tcal=(\Ofrak_K^+,\Nbb,C_\Qbb\cap\Ofrak_K).$$
Note that if we take $K=\Qbb$ then $\Ofrak_K=\Zbb$, $\Ofrak_K^+=\Nbb$, and $C_\Qbb=\Nbb$,
so $C_\Qbb\cap\Ofrak_K=\Nbb$ (i.e., it assigns to any such collection $\al_1,\dots,\al_n$ the submonoid $\Nbb\subseteq\Zbb$). Hence this Frobenius template reduces down to the classical Frobenius template when $K=\Qbb$. We now arrive at the first main theorem of this paper.

\begin{restatable}[]{theorem}{FrobTemplate}
\label{FrobTemplate}
Let $K$ be a real number field. Then the template $\Tcal=(\Ofrak_K^+,\Nbb,C_\Qbb\cap\Ofrak_K)$ is Frobenius.
\end{restatable}

\noindent
After proving this, we begin to look at the structure of the Frobenius semigroup $\Frob(\al_1,\dots,\al_n)$ in the template $(\Ofrak_K^+,\Nbb,C_\Qbb\cap\Ofrak_K)$. In particular, we define a partial ordering on $\Frob(\al_1,\dots,\al_n)$ and show that $\Frob(\al_1,\dots,\al_n)$ contains maximal elements with respect to this ordering. Furthermore, we show that the set $\Mfrak(\al_1,\dots,\al_n)$ of all such maximal elements satisfies the following:

\begin{restatable}[]{theorem}{FinitelyManyMaximalElements}\label{FinitelyManyMaximalElements}
Let $K$ be a real number field and $\al_1,\dots,\al_n\in\Ofrak_K^+$ be nonzero elements that generate $\Ofrak_K$ as a $\Zbb$-module. Then $\Mfrak(\al_1,\dots,\al_n)$ is a finite set, and $\Frob(\al_1,\dots,\al_n)$ is equal to the finite union
    $$\Frob(\al_1,\dots,\al_n)=\bigcup_{\mu\in\Mfrak(\al_1,\dots,\al_n)}(\mu+C_\Qbb(\al_1,\dots,\al_n)\cap\Ofrak_K).$$
\end{restatable}

\noindent
After establishing these properties of the Frobenius semigroups, we give an explicit calculation of certain Frobenius semigroups for real quadratic number fields. Lastly, we show the remarkable result that in general, the size of $\Mfrak(\al_1,\dots,\al_n)$ can be unbounded, even if $n$ is fixed.

\section{Some Preliminary Results}\label{Some preliminary results section}
For completeness, we give a summary of results from \cite{integerknaps} that are used in this paper, and also some general results about matrices that will be used later on. Let $A\in\Zbb^{d\times n}$, $1\leqslant d< n$, be a matrix satisfying
\begin{enumerate}[label=(I\arabic*), leftmargin=\widthof{[Step-]}+\labelsep]
    \item\label{I1} $\gcd\{\det(A') \mid A' \text{ is a } d\times d \text{ minor of } A\}=1$;
    \item\label{I2} $\left\{\left.x\in\Rbb_{\geqslant0}^n \, \right|  Ax=0\right\}=0.$
\end{enumerate}
Let $\Fcal(A)\subseteq\Zbb^d$ denote the set
    $$\Fcal(A)=\{Ax \mid x\in\Nbb^n\},$$
let
    $$C_\Rbb(A)=\left\{Ax \left| \, x\in\Rbb^n_{\geqslant0}\!\right.\right\}$$
be the positive cone generated by the columns of $A$, and  similarly let 
    $$C_\Qbb(A)=\left\{Ax \left| \, x\in\Qbb^n_{\geqslant0}\!\right.\right\}$$
be the positive rational cone generated by the columns of $A$. Then the following is a slightly weaker version of Lemma 1.1 in \cite{integerknaps}.
\begin{lemma}[Lemma 1.1, \cite{integerknaps}]\label{lemma 1.1, [AH10]}
Let $A\in\Zbb^{d\times n}$, $1\leqslant d< n$, be an integral matrix satisfying conditions \ref{I1} and \ref{I2}. Then  for any integer vector  $w\in\inte(C_\Rbb(A))\cap\Zbb^d$ (where $\inte(C_\Rbb(A))=\{Ax \mid x\in\Rbb^n_{>0}\}$), there is some positive number $N\geqslant0$ so that if $t\geqslant N$, then 
    $$(tw+C_\Rbb(A))\cap\Zbb^d\subseteq\Fcal(A).$$
\end{lemma}
\noindent
Note that if we take $t\in\Nbb$ sufficiently large in the above lemma, then $tw\in\Zbb^d$, so Lemma \ref{lemma 1.1, [AH10]} shows that we will have 
    $$tw+C_\Rbb(A)\cap\Zbb^d\subseteq\Fcal(A)$$
for all large enough positive integers $t$. \\
\\
The following lemma, whose proof is based on \cite{133077}, will be important in the proof of Theorem \ref{FrobTemplate}.

\begin{lemma}\label{surjective implies gcd of det of minors is 1} 
Let $R$ be a commutative ring with unity and $A\in R^{d\times n}$ be a matrix for which the corresponding $R$-module homomorphism $R^n\to R^d$ is surjective. If $A_1,\dots,A_N$ denote the $d\times d$ minors of $A$ (i.e., the matrices resulting from selecting $d$ distinct columns of $A$), then $\det(A_1),\dots,\det(A_N)$ generate the unit ideal in $R$. In particular, if $R=\Zbb$ then $\gcd(\det(A_1),\dots,\det(A_N))=1$.  
\end{lemma}
\begin{proof}
Because the $R$-module homomorphism $R^n\to R^d$ corresponding to $A$ is surjective, and because free $R$-modules are projective, we know that there is an $R$-module homomorphism $R^d\to R^n$ for which 
\begin{center}
    \begin{tikzcd}[column sep=50, row sep=40]
        & \arrow[dl, dashed]R^d \arrow[d]& \\
        R^n \arrow[r, "\textstyle A"] & R^d \arrow[r] & 0
    \end{tikzcd}
\end{center}
commutes, where the map $R^d\to R^d$ is the identity. Hence there is a matrix $B\in R^{n\times d}$ corresponding to the map $R^d\to R^n$ for which $AB=I$ is the $d\times d$ identity matrix. If $B_1,\dots,B_N$ denote the $d\times d$ minors of $B$ (i.e., the matrices resulting from selecting $d$ distinct rows of $B$), then the Cauchy-Binet formula shows that 
    $$1=\det(I)=\det(AB)=\sum_{i=1}^{N}\det(A_i)\det(B_i),$$
so $\det(A_1)R+\cdots+\det(A_N)R=R$. \\
\\
If $R=\Zbb$ then we know that $\det(A_1)\Zbb+\cdots+\det(A_N)\Zbb=\Zbb$, so it follows that $\gcd(\det(A_1),\dots,\det(A_N))=1$. 
\end{proof}

\section{The Frobenius Problem}\label{The Frobenius Problem Section}
Using the results of section \ref{Some preliminary results section}, we can prove the first main theorem of this paper. This proof is based on ideas present in the proof of Theorem 1.1 in \cite{TotallyReal}.
\FrobTemplate*
\begin{proof}
Fix elements $\al_1,\dots,\al_n\in\Ofrak_K^+$ that generate $\Ofrak_K$ as a $\Zbb$-module. Then we know that $n\geqslant[K:\Qbb]$, which is the rank of $\Ofrak_K$ as a $\Zbb$-module. We first deal with the case where $n=[K:\Qbb]$, meaning $\al_1,\dots,\al_n$ span the free $\Zbb$-module $\Ofrak_K$ of rank $n$, and are thus a basis for $\Ofrak_K$ as a $\Zbb$-module. In this case, if $\be\in C_\Qbb(\al_1,\dots,\al_n)\cap\Ofrak_K$, we have
    $$\be=\sum_{i=1}^{n}x_i\al_i=\sum_{i=1}^{n}y_i\al_i,$$
where $x_1,\dots,x_n\in\Qbb_{\geqslant0}$ and $y_1,\dots,y_n\in\Zbb$. Then the fact that $K$ is the field of fractions of $\Ofrak_K$, and that $\Qbb$ is a flat $\Zbb$-module, shows that $\al_1,\dots,\al_n$ is also a basis for $K$ as a $\Qbb$-vector space, and in particular, they are $\Qbb$-linearly independent. Hence each $x_i=y_i$, so each $x_i\in\Nbb$, which means that $\be\in\SG(\al_1,\dots,\al_n)$, and thus
    $$C_\Qbb(\al_1,\dots,\al_n)\cap\Ofrak_K=\SG(\al_1,\dots,\al_n).$$
Then $\Frob(\al_1,\dots,\al_n)=\SG(\al_1,\dots,\al_n)$, so it is nonempty. \\
\\
Now suppose that $[K:\Qbb]\coloneqq d<n$, and let $\be_1,\dots,\be_d\in\Ofrak_K$ be a basis for $\Ofrak_K$ as a $\Zbb$-module. Let $a_{ij}\in\Zbb, i=1,\dots,n,j=1,\dots,d$, be the unique integers for which
    $$\al_i=\sum_{j=1}^{d}a_{ij}\be_j.$$
Let $\vp:\Ofrak_K\to\Zbb^d$ be the isomorphism associated to the $\Zbb$-basis $\be_1,\dots,\be_d$, so $\vp(\al_i)=(a_{i1},\dots,a_{id})$, and let $A=(a_{ij})^\mathrm{T}\in\Zbb^{d\times n}$ be the matrix whose columns are the $\vp(\al_i)$.
Note that $\vp:\Ofrak_K\to\Zbb^d$ is the restriction of the vector space isomorphism $K\to\Qbb^d$ associated to the same basis $\be_1,\dots,\be_d$ for $K$ as a $\Qbb$-vector space.
Let $\si_1,\dots,\si_d:K\hookrightarrow\Cbb$ be the $d$ distinct embeddings of $K$ into $\Cbb$, and let $\si:K\to\Cbb^d$ be the Minkowski embedding of $K$ into $\Cbb^d$, given by 
    $$\si(x)=(\si_1(x),\dots,\si_d(x))$$
for $x\in K$. Note that $\si$ is a $\Qbb$-linear map because each $\si_i$ must fix $\Qbb$. Let $B=(\si_i(\be_j))\in\Cbb^{d\times d}$ and $C=(\si_i(\al_j))\in\Cbb^{d\times n}$ be the matrices whose columns are the vectors $\si(\be_1),\dots,\si(\be_d)\in\Cbb^d$ and $\si(\al_1),\dots,\si(\al_n)\in\Cbb^d$, respectively. Then
    $$(BA)_{ij}=\sum_{k=1}^{d}B_{ik}A_{kj}=\sum_{k=1}^{d}\si_i(\be_k)a_{jk}=\si_i\!\left(\sum_{k=1}^{d}a_{jk}\be_k\right)=\si_i(\al_j)=C_{ij},$$
so $BA=C$. If we let $\si_1:K\hookrightarrow\Cbb$ be the inclusion of $K$ into $\Cbb$, then note that the first row of $C$ will contain all real positive entries because $\al_1,\dots,\al_n\in\Ofrak_K^+$.  \\
\\
We now claim that the matrix $A\in\Zbb^{d\times n}$ satisfies conditions \ref{I1} and \ref{I2}. Because $\al_1,\dots,\al_n$ generate $\Ofrak_K$ as a $\Zbb$-module, we know that $\Zbb\al_1+\cdots+\Zbb\al_n=\Ofrak_K$, so 
    $$\vp(\Ofrak_K)=\Zbb\vp(\al_1)+\cdots+\Zbb\vp(\al_n)=A\Zbb^n,$$
where we used the fact that the columns of $A\in\Zbb^{d\times n}$ are the $\vp(\al_1),\dots,\vp(\al_n)$. But $\vp:\Ofrak_K\to\Zbb^d$ is an isomorphism, so $\vp(\Ofrak_K)=\Zbb^d$, and thus $A\Zbb^n=\Zbb^d$, so the linear transformation $\Zbb^n\to\Zbb^d$ associated to the matrix $A$ is surjective. 
Lemma \ref{surjective implies gcd of det of minors is 1} then shows that 
    $$\gcd\{\det(A') \ | \ A' \text{ is a } d\times d \text{ minor of } A\}=1,$$
so condition \ref{I1} is satisfied. Now suppose that $x=(x_1,x_2,\dots,x_n)\in\Rbb_{\geqslant0}^n$ is such that $Ax=0$. Then
    $$BAx=Cx=\begin{pmatrix}
    \al_1 & \cdots & \al_n\\
    \si_2(\al_1) & \cdots & \si_2(\al_n)\\
    \vdots & \ddots & \vdots\\
    \si_d(\al_1) & \cdots & \si_d(\al_n)
    \end{pmatrix}\begin{pmatrix}
    x_1\\
    x_2\\
    \vdots\\
    x_n
    \end{pmatrix}=\begin{pmatrix}
    x_1\al_1+\cdots+x_n\al_n\\
    x_1\si_2(\al_1)+\cdots+x_n\si_2(\al_n)\\
    \vdots\\
    x_1\si_d(\al_1)+\cdots+x_n\si_d(\al_n)
    \end{pmatrix}=\begin{pmatrix}
    0\\
    0\\
    \vdots\\
    0
    \end{pmatrix}$$
(recall we are taking $\si_1$ to be the inclusion of $K$ into $\Cbb$). Because $\al_1,\dots,\al_n\in\Ofrak_K^+$ are nonzero, we know that each $\al_i>0$, so the fact that the $x_i\geqslant0$ shows that in order for $x_1\al_1+\dots+x_n\al_n=0$, we must have that each $x_i=0$. Hence
    $$\left\{\left.x\in\Rbb_{\geqslant0}^{n} \,\right| Ax=0\right\}=0,$$
so the matrix $A$ also satisfies condition \ref{I2}. \\
\\
We now apply Lemma \ref{lemma 1.1, [AH10]} to show that if $\ga\in\inte(C_\Qbb(\al_1,\dots,\al_n))\cap\Ofrak_K$ (where $\inte(C_\Qbb(\al_1,\dots,\al_n))$ denotes strictly positive rational linear combinations of $\al_1,\dots,\al_n$), then there is some nonzero $t\in\Nbb$ for which 
    $$t\ga+C_\Qbb(\al_1,\dots,\al_n)\cap\Ofrak_K\subseteq\SG(\al_1,\dots,\al_n).$$
If this is true, then we can take $\ga=\sum\limits_{i=1}^{n}x_i\al_i$, where $x_1,\dots,x_n\in\Nbb$ are nonzero, so $\ga\in\inte(C_\Qbb(\al_1,\dots,\al_n))\cap\Ofrak_K$ and $t\ga\in\SG(\al_1,\dots,\al_n)$, and thus $t\ga\in\Frob(\al_1,\dots,\al_n)$, so $\Frob(\al_1,\dots,\al_n)$ is nonempty. Then it would follow that for any $\al_1,\dots,\al_n\in\Ofrak_K^+$ that generate $\Ofrak_K$ as a $\Zbb$-module, the set $\Frob(\al_1,\dots,\al_n)$ is nonempty, so we would know that the template $\Tcal=\left(\Ofrak_K^+,\Nbb,C_\Qbb\cap\Ofrak_K\right)$ is Frobenius. \\
\\
In order to apply Lemma \ref{lemma 1.1, [AH10]} in this manner, we must first transition from the situation we have in $\Ofrak_K$ to the situation given in Lemma \ref{lemma 1.1, [AH10]}. Because $\vp(\al_1),\dots,\vp(\al_n)\in\Zbb^d$ are the columns of the matrix $A$, we know that 
    $$\vp(C_\Qbb(\al_1,\dots,\al_n))=C_\Qbb(A),\qquad \vp(\SG(\al_1,\dots,\al_n))=\Fcal(A),$$
and the fact that $\vp:\Ofrak_K\to\Zbb^d$ is a $\Zbb$-module isomorphism shows that $\vp(\Ofrak_K)=\Zbb^d$. It follows that if $t\in\Nbb$ and $\ga\in\inte(C_\Qbb(\al_1,\dots,\al_n))\cap\Ofrak_K$, then
\begin{align*}
    \vp(t\ga+C_\Qbb(\al_1,\dots,\al_n)\cap\Ofrak_K)&=t\vp(\ga)+\vp(C_\Qbb(\al_1,\dots,\al_n))\cap\vp(\Ofrak_K)=t\vp(\ga)+C_\Qbb(A)\cap\Zbb^d.
\end{align*}
Hence 
    $$t\ga+C_\Qbb(\al_1,\dots,\al_n)\cap\Ofrak_K\subseteq\SG(\al_1,\dots,\al_n)$$
if and only if 
    $$t\vp(\ga)+C_\Qbb(A)\cap\Zbb^d\subseteq\Fcal(A).$$
But the fact that $\ga\in\inte(C_\Qbb(\al_1,\dots,\al_n))\cap\Ofrak_K$ shows that $\vp(\ga)\in\inte(C_\Qbb(A))\cap\Zbb^d\subseteq\inte(C_\Rbb(A))\cap\Zbb^d$, so Lemma \ref{lemma 1.1, [AH10]} shows that for all sufficiently large $t\in\Nbb$, 
    $$t\vp(\ga)+C_\Qbb(A)\cap\Zbb^d\subseteq t\vp(\ga)+C_\Rbb(A)\cap\Zbb^d\subseteq\Fcal(A),$$
and thus
    \[t\ga+C_\Qbb(\al_1,\dots,\al_n)\cap\Ofrak_K\subseteq\SG(\al_1,\dots,\al_n).\qedhere\]
\end{proof}
\noindent
Now that we have shown whenever $\al_1,\dots,\al_n\in\Ofrak_K^+$ generate $\Ofrak_K$ as a $\Zbb$-module, the Frobenius semigroup
    $$\Frob(\al_1,\dots,\al_n)=\{w\in\SG(\al_1,\dots,\al_n) \mid w+C_\Qbb(\al_1,\dots,\al_n)\cap\Ofrak_K\subseteq\SG(\al_1,\dots,\al_n)\}$$
is nonempty, we can begin to describe it in certain cases. The simplest case is when $\al_1,\dots,\al_n$ are linearly independent, in which case they form an integral basis for $\Ofrak_K$. In this case, the beginning of the proof of Theorem \ref{FrobTemplate} shows that the following is true. 

\begin{lemma}\label{Frob of basis}
Let $K$ be a real number field of degree $d$, and suppose that $\be_1,\dots,\be_d\in\Ofrak_K^+$ is a basis for $\Ofrak_K$ as a $\Zbb$-module. Then
    $$\Frob(\be_1,\dots,\be_d)=\SG(\be_1,\dots,\be_d).$$
\end{lemma}

\noindent
The following lemma is an immediate consequence of the definitions. 
\begin{lemma}\label{al in SG implies Frobs ar equal}
Let $K$ be a real number field, and suppose that $\al_1,\dots,\al_n\in\Ofrak_K^+$ generate $\Ofrak_K$ as a $\Zbb$-module. Then if $\al\in\SG(\al_1,\dots,\al_n)$,
    $$\Frob(\al_1,\dots,\al_n,\al)=\Frob(\al_1,\dots,\al_n).$$
\end{lemma}
\begin{proof}
If $\al\in\SG(\al_1,\dots,\al_n)$ then we know that $\SG(\al_1,\dots,\al_n,\al)=\SG(\al_1,\dots,\al_n)$ and $C_\Qbb(\al_1,\dots,\al_n,\al)=C_\Qbb(\al_1,\dots,\al_n)$. Hence $C_\Qbb(\al_1,\dots,\al_n,\al)\cap\Ofrak_K=C_\Qbb(\al_1,\dots,\al_n)\cap\Ofrak_K$, so the definition of the Frobenius semigroup shows that 
    \[\Frob(\al_1,\dots,\al_n,\al)=\Frob(\al_1,\dots,\al_n).\qedhere\]
\end{proof}

\noindent
The converse to Lemma \ref{al in SG implies Frobs ar equal} will be true provided that the elements $\al_1,\dots,\al_n$  form a basis for $\Ofrak_K$.

\begin{corollary}
Let $K$ be a real number field and $\be_1,\dots,\be_d\in\Ofrak_K^+$ be a basis for $\Ofrak_K$ as a $\Zbb$-module. If $\al\in\Ofrak_K^+$, then 
    $$\Frob(\be_1,\dots,\be_d,\al)=\Frob(\be_1,\dots,\be_d)$$
if and only if $\al\in\SG(\be_1,\dots,\be_d)$.
\end{corollary}
\begin{proof}
The `if' part is handled by Lemma \ref{al in SG implies Frobs ar equal}, so suppose that $\Frob(\be_1,\dots,\be_d,\al)=\Frob(\be_1,\dots,\be_d)$. Then by Lemma \ref{Frob of basis}, we know that $\Frob(\be_1,\dots,\be_d,\al)=\SG(\be_1,\dots,\be_d)$, so in particular, $0\in\Frob(\be_1,\dots,\be_d,\al)$, and thus $\Frob(\be_1,\dots,\be_d,\al)=\SG(\be_1,\dots,\be_d,\al)$. It follows that  $\SG(\be_1,\dots,\be_d,\al)=\SG(\be_1,\dots,\be_d)$, so we have that $\al\in\SG(\be_1,\dots,\be_d)$. 
\end{proof}
\noindent
It is not always true that
    $$\Frob(\al_1,\dots,\al_n)\subseteq\Frob(\al_1,\dots,\al_n,\al)$$
for $\al_1,\dots,\al_n,\al\in\Ofrak_K^+$. Suppose that $\be_1,\dots,\be_d\in\Ofrak_K^+$ is a basis for $\Ofrak_K$ as a $\Zbb$-module, and $\al\in\Ofrak_K^+$ is such that $0\notin\Frob(\be_1,\dots,\be_d,\al)$. Then $0\in\Frob(\be_1,\dots,\be_d)$, so $\Frob(\be_1,\dots,\be_d)\not\subseteq\Frob(\be_1,\dots,\be_d,\al)$. See section \ref{An Example of a Simple Frobenius Set For Quadratic Extensions Section} for an explicit example of a real number field $K$ and such an $\al\in\Ofrak_K^+$, as well as an example of elements $\al_1,\dots,\al_n\in\Ofrak_K^+$ that generate $\Ofrak_K$ as a $\Zbb$-module, are not a basis for $\Ofrak_K$, and for which $\Frob(\al_1,\dots,\al_n)=\SG(\al_1,\dots,\al_n)$.

\section{The Structure of the Frobenius Set}\label{The Structure of the Frobenius Set Section}
Let $K$ be a real number field and fix $\al_1,\dots,\al_n\in\Ofrak_K^+$ that generate $\Ofrak_K$ as a $\Zbb$-module. Define a partial ordering $\preccurlyeq$ on $\Frob(\al_1,\dots,\al_n)$ by $w\preccurlyeq v$ for $w,v\in\Frob(\al_1,\dots,\al_n)$ if and only if 
    $$w+C_\Qbb(\al_1,\dots,\al_n)\cap\Ofrak_K\subseteq v+C_\Qbb(\al_1,\dots,\al_n)\cap\Ofrak_K.$$
Note that if $w\preccurlyeq v$ then $w=v+\ga$ for some $\ga\in C_\Qbb(\al_1,\dots,\al_n)\cap\Ofrak_K$, so the fact that every element of $C_\Qbb(\al_1,\dots,\al_n)\cap\Ofrak_K$ is non-negative implies that $w\geqslant v$, where $\geqslant$ is the standard total order on $\Rbb$. 
The relation $\preccurlyeq$ is clearly reflexive and transitive, and the previous observation implies that $\preccurlyeq$ is also antisymmetric. This relation is also compatible with addition in $\Frob(\al_1,\dots,\al_n)$, in the sense that if $u,v,w\in\Frob(\al_1,\dots,\al_n)$ and $u\preccurlyeq v$, then $u+w\preccurlyeq v+w$. Furthermore, note that if $v\in\Frob(\al_1,\dots,\al_n)$ and $w\in v+C_\Qbb(\al_1,\dots,\al_n)\cap\Ofrak_K$, then $w=v+\ga$ for some $\ga\in C_\Qbb(\al_1,\dots,\al_n)\cap\Ofrak_K$, and thus 
\begin{align*}
    w+C_\Qbb(\al_1,\dots,\al_n)\cap\Ofrak_K&=v+\ga+C_\Qbb(\al_1,\dots,\al_n)\cap\Ofrak_K\\
    &\subseteq v+C_\Qbb(\al_1,\dots,\al_n)\cap\Ofrak_K\subseteq\SG(\al_1,\dots,\al_n),
\end{align*}
which shows that both $w\in\Frob(\al_1,\dots,\al_n)$ and $w\preccurlyeq v$. Thence $\Frob(\al_1,\dots,\al_n)$ will not contain any minimal elements with respect to $\preccurlyeq$, since if $v\in\Frob(\al_1,\dots,\al_n)$ then we can choose any 
    $$w\in v+C_\Qbb(\al_1,\dots,\al_n)\cap\Ofrak_K\subseteq\Frob(\al_1,\dots,\al_n)$$
that is distinct from $v$, and then we will have $w\preccurlyeq v$ but $w\neq v$. However, we claim that $\Frob(\al_1,\dots,\al_n)$ has maximal elements with respect to the partial ordering $\preccurlyeq$, and furthermore, there are only finitely many such maximal elements. 

\begin{lemma}\label{max elements of Frob}
Let $K$ be a real number field and fix nonzero $\al_1,\dots,\al_n\in\Ofrak_K^+$ that generate $\Ofrak_K$ as a $\Zbb$-module. Then $\Frob(\al_1,\dots,\al_n)$ contains an element that is maximal with respect to $\preccurlyeq$. Furthermore, if $w\in\Frob(\al_1,\dots,\al_n)$, then $w\preccurlyeq\mu$ for some maximal element $\mu\in\Frob(\al_1,\dots,\al_n)$. 
\end{lemma}
\begin{proof}
For brevity, set $C=C_\Qbb(\al_1,\dots,\al_n)\cap\Ofrak_K, S=\SG(\al_1,\dots,\al_n),$ and $F=\Frob(\al_1,\dots,\al_n)$. We first claim that for any $w\in S$, there are finitely many $v\in S$ for which $w\geqslant v$. Suppose that this is not the case, so there are an infinite number of distinct elements $v_1,v_2,\ldots\in S$ for which $w\geqslant v_i$ for every $i=1,2,\dots$. Let $v_{ij}\in\Nbb$ be (not necessarily unique) natural numbers for which
    $$v_i=\sum_{j=1}^{n}v_{ij}\al_j.$$
If there is some integer $N>0$ such that $v_{ij}\leqslant N$ for all $i=1,2,\dots$ and $j=1,\dots,n$, then there could be at most $(N+1)^n$ choices for the $v_i$, since the coefficient of each $\al_j$ in some representation of $v_i$ as a sum of the $\al_j$ must be one of the natural numbers $0,1,\dots,N$. Hence the coefficients $v_{ij}$ must grow without bound as $i\to\infty$ and $j$ ranges from $1$ to $n$. The fact that all of the $\al_j$ are positive then shows that at least one $v_i$ must eventually be bigger than $w$, contradicting the fact that $w\geqslant v_i$ for all $i$. Hence there are finitely many $v\in S$ for which $w\geqslant v$. If $w,v\in F\subseteq S$ and $w\preccurlyeq v$, then $w\geqslant v$, so it also follows that for any $w\in F$, there are only finitely many $v\in F$ for which $w\preccurlyeq v$.\\
\\
We now apply Zorn's lemma to show that $F$ has a maximal element with respect to $\preccurlyeq$. If $\Si\subseteq F$ is a finite chain in $F$, then $\Si$ clearly has an upper bound in $F$ with respect to $\preccurlyeq$, so suppose that $\Si\subseteq F$ is an infinite chain in $F$. Furthermore, assume that $\Si$ does not have an upper bound in $F$. Then for any $w\in\Si$, we can find some $v_1\in\Si\setminus\{w\}$ for which $w\preccurlyeq v_1$; else $v_1\preccurlyeq w$ for all $v_1\in\Si\setminus\{w\}$ because $\Si$ is totally ordered, and thus $w$ will be an upper bound of $\Si$. Similarly, we can find some $v_2\in\Si\setminus\{w,v_1\}$ for which $w\preccurlyeq v_1\preccurlyeq v_2$; else $v_2\preccurlyeq v_1$ for all $v_2\in\Si\setminus\{w,v_1\}$ because $\Si$ is totally ordered, and thus $v_1$ will be an upper bound of $\Si$. Continuing in this manner, we can find an infinite number of distinct elements $v_1,v_2,\ldots\in\Si\subseteq F$ for which $w\preccurlyeq v_1\preccurlyeq v_2\preccurlyeq\cdots$, which contradicts the fact that there are only finitely many $v\in F$ for which $w\preccurlyeq v$. Hence $\Si$ must have an upper bound, which means that we can apply Zorn's lemma to conclude that $F$ has a maximal element with respect $\preccurlyeq$. \\
\\
Now suppose that $w\in F$, let $F_w\subseteq F$ be the set
    $$F_w=\{v\in F \ | \ w\preccurlyeq v\},$$
and give $F_w$ the induced ordering from $\preccurlyeq$ on $F$. Note that by the observations in the first paragraph of this proof, $\#F_w<\infty$, and $F_w\neq\es$ because $w\in F_w$. Furthermore, if $\mu$ is a maximal element of $F_w$, then $\mu$ is a maximal element of $F$ because if $v\in F\setminus\{\mu\}$ and $\mu\preccurlyeq v$ then $w\preccurlyeq \mu\preccurlyeq v$, and thus $v\in F_w$, so $\mu\preccurlyeq v$ is not possible by the maximality of $\mu$. The fact that $\#F_w<\infty$ implies that any chain in $F_w$ is finite and thus has an upper bound, so Zorn's lemma implies that $F_w$ has a maximal element $\mu$. Then $\mu$ is also a maximal element of $F$, and $w\preccurlyeq \mu$.  
\end{proof}

\noindent
Now define $\Mfrak(\al_1,\dots,\al_n)\subseteq\Frob(\al_1,\dots,\al_n)$ to be the set of all maximal elements of $\Frob(\al_1,\dots,\al_n)$ with respect to $\preccurlyeq$, and note that $\Mfrak(\al_1,\dots,\al_n)\neq\es$ by Lemma \ref{max elements of Frob}. If we combine all the statements of Lemma \ref{max elements of Frob}, then we arrive at the following nice characterization of $\Frob(\al_1,\dots,\al_n)$.

\begin{lemma}\label{Frob as bad union}
Let $K$ be a real number field and $\al_1,\dots,\al_n\in\Ofrak_K^+$ be nonzero elements that generate $\Ofrak_K$ as a $\Zbb$-module. Then 
    $$\Frob(\al_1,\dots,\al_n)=\bigcup_{\mu\in\Mfrak(\al_1,\dots,\al_n)}(\mu+C_\Qbb(\al_1,\dots,\al_n)\cap\Ofrak_K).$$
\end{lemma}
\begin{proof}
Let 
    $$C=C_\Qbb(\al_1,\dots,\al_n)\cap\Ofrak_K,\quad S=\SG(\al_1,\dots,\al_n),\quad F=\Frob(\al_1,\dots,\al_n),\quad\Mfrak=\Mfrak(\al_1,\dots,\al_n).$$
The inclusion
    $$F\supseteq\bigcup_{\mu\in\Mfrak}(\mu+C)$$
is obvious, since if $\mu\in F$ then $\mu+C\subseteq F$. Now, note that Lemma \ref{max elements of Frob} shows that if $w\in F$ then there is some $\mu\in\Mfrak$ for which $w\preccurlyeq \mu$, and thus $w\in \mu+C$, so we immediately get the inclusion in the other direction.
\end{proof}

\noindent
The above lemma provides the first main ingredient in the proof of Theorem \ref{FinitelyManyMaximalElements}. If $n\geqslant 1$, recall that the pointwise partial order $\leqslant_{\mathrm{p}}$ on $\Zbb^n$ is defined by $(x_1,\dots,x_n)\leqslant_\mathrm{p} (y_1,\dots,y_n)$ if and only if $x_i\leqslant y_i$ for all $i=1,\dots,n$. Then Dickson's lemma says that the following is true:

\begin{lemma}[Dickson's Lemma]
Let $n\geqslant 1$. Then any nonempty set $S\subseteq\Nbb^n$ has a finite and nonzero number of minimal elements with respect to $\leqslant_\mathrm{p}$. 
\end{lemma}
\begin{proof}
Let $k$ be a field and consider the polynomial ring $k[x_1,\dots,x_n]$ and the ideal $\afrak=\langle x^s \mid s\in S\rangle\subseteq k[x_1,\dots,x_n]$, where we are using the notation $x^\al=x_1^{\al_1}\cdots x_n^{\al_n}$ when $\al=(\al_1,\dots,\al_n)\in\Nbb^n$. If $S$ is finite then the claim of the lemma  is obvious, so suppose that $S$ is infinite. Then because $\Nbb^n$ is countable, we can write $S=\{s_1,s_2,\dots\}$. By Hilbert's Basis Theorem, $k[x_1,\dots,x_n]$ is a Noetherian ring, so there is some smallest nonzero $m\in\Nbb$ for which 
    $$\langle x^{s_1},\dots,x^{s_m}\rangle=\langle x^{s_1},\dots,x^{s_{m+1}}\rangle=\cdots,$$ 
and thus $\afrak=\langle x^{s_1},\dots,x^{s_m}\rangle$. Then if $s\in S$ is not one of the $s_1,\dots,s_m$, we have $x^s\in \afrak$, so there are $f_1,\dots,f_m\in k[x_1,\dots,x_n]$ for which 
    $$x^s=\sum_{i=1}^{m}f_ix^{s_i}.$$
But this means that at least one monomial term in one of the $f_i$ must be some nonzero element of $k$ times $x^{s-s_i}$, so $0\leqslant_\mathrm{p}s-s_i$, and thus $s_i\leqslant_\mathrm{p}s$. Hence $s$ cannot be a minimal element of $S$, so $s_1,\dots,s_m$ are all the possible minimal elements of $S$. It follows that $S$ can have only finitely many minimal elements, and choosing the minimal elements (with respect to $\leqslant_\mathrm{p}$) out of the $s_1,\dots,s_m$ shows that $S$ has at least one minimal element. 
\end{proof}

\noindent
Using Dickson's lemma, we can show that the set $\Mfrak(\al_1,\dots,\al_n)$ will always have finitely many elements, meaning $\Frob(\al_1,\dots,\al_n)$ will always have finitely many maximal elements with respect to $\preccurlyeq$. This allows us to prove Theorem \ref{FinitelyManyMaximalElements}.

\FinitelyManyMaximalElements*
\begin{proof}
As before, set 
    $$C=C_\Qbb(\al_1,\dots,\al_n)\cap\Ofrak_K,\quad S=\SG(\al_1,\dots,\al_n),\quad F=\Frob(\al_1,\dots,\al_n),\quad\Mfrak=\Mfrak(\al_1,\dots,\al_n).$$
Because $\al_1,\dots,\al_n$ span $\Ofrak_K$ by assumption, we have a surjective $\Zbb$-module homomorphism $f:\Zbb^n\to\Ofrak_K$ given by  
    $$f(x_1,\dots,x_n)=\sum_{i=1}^{n}x_i\al_i.$$
Define a set $W=f^{-1}(F)\cap\Nbb^n\subseteq\Nbb^n$, and note that for any $w\in F$, there is some $w^*\in W$ for which $f(w^*)=w$, since every $w\in F$ must also be in $S=f(\Nbb^n)$. Theorem \ref{FrobTemplate} then shows that $W$ is nonempty because $F$ is. We now claim that if $\mu\in\Mfrak$ and $\mu^*\in f^{-1}(\mu)\cap\Nbb^n$, then $\mu^*$ is a minimal element of $W$ with respect to the pointwise partial ordering $\leqslant_\mathrm{p}$. Suppose that $w^*\in W$ is such that $f(w^*)=w\in F$ and $w^*\leqslant_\mathrm{p} \mu^*$. Let $w^*=(w_1,\dots,w_n)$ and $\mu^*=(\mu_1,\dots,\mu_n)$, so the fact that $w^*\leqslant_\mathrm{p} \mu^*$ implies that each $w_i\leqslant \mu_i$, and thus there are $y_1,\dots,y_n\in\Nbb$ for which $\mu_i=w_i+y_i$. Hence 
    $$w+\sum_{i=1}^{n}y_i\al_i=f((w_1,\dots,w_n)+(y_1,\dots,y_n))=f(\mu_1,\dots,\mu_n)=\mu,$$
so 
    $$\mu+C=w+\sum_{i=1}^{n}y_i\al_i+C\subseteq w+C,$$
where we used the fact that each $y_i\in\Nbb$, so each $y_i\al_i\in S\subseteq C$. But $w\in F$, so the maximality of $\mu$ implies that $\mu=w$, and thus each $y_i=0$ because each $\al_i>0$, so the only way for $\sum\limits_{i=1}^{n}y_i\al_i=0$ when $y_1,\dots,y_n\in\Nbb$ is for all of $y_1,\dots,y_n$ to be zero. Then $\mu^*=w^*$, which means that $\mu^*$ must be a minimal element of $W$ with respect to the pointwise partial ordering $\leqslant_\mathrm{p}$. By Dickson's lemma, there are finitely many minimal elements of $W\subseteq\Nbb^n$, so it follows that there must also be finitely many maximal elements of $F$. Hence $\Mfrak$ is a finite set, and Lemma \ref{Frob as bad union} shows that $F$ may be written as the finite union 
    \[F=\bigcup_{\mu\in\Mfrak}(\mu+C).\qedhere\]
\end{proof}

\noindent
The union in Theorem \ref{FinitelyManyMaximalElements} is interesting since it is the smallest possible way of writing $\Frob(\al_1,\dots,\al_n)$ as a union of translates of the set $C_\Qbb(\al_1,\dots,\al_n)\cap\Ofrak_K$. This is because if $W$ is some collection of elements of $\Frob(\al_1,\dots,\al_n)$ for which 
\begin{equation}
    \Frob(\al_1,\dots,\al_n)=\bigcup_{w\in W}(w+C_\Qbb(\al_1,\dots,\al_n)\cap\Ofrak_K), \label{Frob as union of elements of W}
\end{equation}
then each $\mu\in\Mfrak$ must be contained in some $w+C_\Qbb(\al_1,\dots,\al)\cap\Ofrak_K$, so $\mu=w$ for some $w\in W$ and thus $\Mfrak\subseteq W$. It turns out that if we add a few more conditions to the set $W$, then we can get a valuable characterization of when a given set of elements in $\Frob(\al_1,\dots,\al_n)$ will be the set of maximal elements. Even more so, this characterization allows us to avoid having to actually check the maximality of certain elements in many explicit calculations of the Frobenius semigroup.

\begin{lemma}\label{condition for maximality}
Let $K$ be a real number field and $\al_1,\dots,\al_n\in\Ofrak_K^+$ be distinct nonzero elements that generate $\Ofrak_K$ as a $\Zbb$-module. Suppose that $\mu_1,\dots,\mu_m\in\Frob(\al_1,\dots,\al_n)$ satisfy
\begin{enumerate}[label=(\arabic*)]
    \item\label{maxcond1} for all $w\in\Frob(\al_1,\dots,\al_n)$, there is some $i=1,\dots,m$ for which $w\preccurlyeq\mu_i$;
    \item\label{maxcond2} for all distinct $i,j\in\{1,\dots,m\}$, $\mu_i\gprecneq\mu_j$. 
\end{enumerate}
Then $\Mfrak(\al_1,\dots,\al_n)=\{\mu_1,\dots,\mu_m\}$.
\end{lemma}
\begin{proof}
As before, let
    $$C=C_\Qbb(\al_1,\dots,\al_n)\cap\Ofrak_K,\quad S=\SG(\al_1,\dots,\al_n),\quad F=\Frob(\al_1,\dots,\al_n),\quad\Mfrak=\Mfrak(\al_1,\dots,\al_n).$$
Then condition \ref{maxcond1} shows that if $\mu\in\Mfrak$ then $\mu\preccurlyeq\mu_i$ for some $i$. Hence $\mu=\mu_i$ by the maximality of $\mu$, so $\Mfrak\subseteq\{\mu_1,\dots,\mu_m\}$. Now fix $i=1,\dots,m$. Then Lemma \ref{max elements of Frob} shows that $\mu_i\preccurlyeq\mu$ for some maximal $\mu\in\Mfrak$, and condition \ref{maxcond1} shows that $\mu\preccurlyeq\mu_j$ for some $j=1,\dots,m$, and thus $\mu=\mu_j$. Hence $\mu_i\preccurlyeq\mu=\mu_j$, so condition \ref{maxcond2} shows that $j=i$ and $\mu=\mu_i$. Hence $\{\mu_1,\dots,\mu_m\}\subseteq\Mfrak$, so $\Mfrak=\{\mu_1,\dots,\mu_m\}.$
\end{proof}

\noindent
An argument identical to the one given in Lemma \ref{condition for maximality} also shows that if $W\subseteq\Frob(\al_1,\dots,\al_n)$ is a set satisfying equation (\ref{Frob as union of elements of W}), and no two elements of $W$ precede one another, then $W$ must be the set of maximal elements $\Mfrak(\al_1,\dots,\al_n)$. In particular, this means that $W$ must be finite. It is thus not possible to write $\Frob(\al_1,\dots,\al_n)$ as an infinite union of translates of the set $C_\Qbb(\al_1,\dots,\al_n)\cap\Ofrak_K$, where no translate is a subset of another one. It is also interesting to consider how the distinct maximal elements of $\Frob(\al_1,\dots,\al_n)$ are related to each other. The next lemma begins to answer this question. 

\begin{lemma}\label{maximal elements are pairwise linearly independent}
Let $K$ be a real number field and $\al_1,\dots,\al_n\in\Ofrak_K^+$ generate $\Ofrak_K$ as a $\Zbb$-module. Then any two distinct elements of $\Mfrak(\al_1,\dots,\al_n)$ are $\Zbb$-linearly independent. 
\end{lemma}
\begin{proof}
Let 
    $$C=C_\Qbb(\al_1,\dots,\al_n)\cap\Ofrak_K,\quad S=\SG(\al_1,\dots,\al_n),\quad F=\Frob(\al_1,\dots,\al_n),\quad\Mfrak=\Mfrak(\al_1,\dots,\al_n),$$
and suppose that this is not true. Then there are distinct $\mu_1,\mu_2\in\Mfrak$ and nonzero $y_1,y_2\in\Zbb$ for which 
    $$y_1\mu_1+y_2\mu_2=0.$$
Because $\mu_1,\mu_2>0$, precisely one of the $y_i$ must be less than zero, so we may assume that $x_1,x_2\in\Nbb$ are nonzero, $x_1<x_2$ (if $x_1=x_2$ then it trivially follows that $\mu_1=\mu_2$), and that we have an equation in the form
    $$x_1\mu_1-x_2\mu_2=0,$$
so $\mu_2=(x_1/x_2)\mu_1$. Then $x_1/x_2<1$, so $1-x_1/x_2\in\Qbb_{\geqslant0}$, and thus $\mu_1-\mu_2=(1-x_1/x_2)\mu_1\in C_\Qbb(\al_1,\dots,\al_n)$ because $\mu_1\in S$ by definition. But $\mu_1-\mu_2\in\Ofrak_K$, so it follows that $\mu_1-\mu_2\in C$, and thus 
    $$\mu_1+C=\mu_2+(\mu_1-\mu_2)+C\subseteq \mu_2+C.$$
The maximality of $\mu_1$ then implies that $\mu_1=\mu_2$, which contradicts the original assumption that $\mu_1,\mu_2\in\Mfrak$ were distinct. 
\end{proof}

\noindent
While it may be tempting to try to generalize this to show that any number of maximal elements of $\Frob(\al_1,\dots,\al_n)$ are linearly independent, results in section \ref{The Number of Maximal Elements Section} show that this fails.

\section{An Example of a Simple Frobenius Semigroup For Quadratic Extensions}\label{An Example of a Simple Frobenius Set For Quadratic Extensions Section}

We now offer an explicit calculation of the Frobenius semigroup for certain collections of three elements in real quadratic number fields. Before we do this, we prove some basic properties about the $C_\Qbb\cap\Ofrak_K$ and $\SG$ sets in this case.

\begin{lemma}\label{C_Q in real quadratic}
Let $K$ be a real quadratic number field with positive integral basis $\be_1,\be_2\in\Ofrak_K^+$, and suppose that $\al=a_2\be_2-a_1\be_1\in\Ofrak_K^+$, where $a_1,a_2\in\Nbb$ are nonzero. Then 
\begin{equation}
    C_\Qbb(\be_1,\be_2,\al)\cap\Ofrak_K=\left\{y_1\be_1+y_2\be_2 \left|\, y_1,y_2\in\Zbb, \ y_2\geqslant 0 \,\text{ and } 
    y_2\geqslant -\frac{a_2}{a_1}y_1\right.\right\}\label{C_Q(b1,b2,a)}.
\end{equation}
\end{lemma}
\begin{proof}
Let $C$ denote the set on the right hand side of equation (\ref{C_Q(b1,b2,a)}), and first suppose that $y_1\be_1+y_2\be_2\in C$, with $y_1\in\Zbb, y_2\in\Nbb$, and $y_2\geqslant -(a_2/a_1)y_1$. Then 
    $$\frac{y_2}{a_2}\al=y_2\be_2-\frac{y_2a_1}{a_2}\be_1.$$
By definition, we know that
    $$\frac{y_2a_1}{a_2}+y_1=\frac{y_2a_1+y_1a_2}{a_2}\geqslant0,$$
so it follows that 
    $$y_2\be_2+y_1\be_1=\frac{y_2}{a_2}\al+\left(\frac{y_2a_1}{a_2}+y_1\right)\!\be_1\in C_\Qbb(\be_1,\be_2,\al)\cap\Ofrak_K,$$
and thus $C\subseteq C_\Qbb(\be_1,\be_2,\al)\cap\Ofrak_K$. Now suppose that $z,z_1,z_2\in\Qbb_{\geqslant0}$ and $z\al+z_1\be_1+z_2\be_2\in\Ofrak_K$. Then 
    $$z\al+z_1\be_1+z_2\be_2=(z_1-za_1)\be_1+(z_2+za_2)\be_2,$$
so $z_1-za_1,z_2+za_2\in\Zbb$ because $\be_1,\be_2$ is a basis for $\Ofrak_K$. Furthermore,
    $$-\frac{a_2}{a_1}(z_1-za_1)=-\frac{a_2z_1}{a_1}+za_2\leqslant z_2+za_2,$$
since $a_1,a_2,z_1,z_2\geqslant0$. Hence $z\al+z_1\be_1+z_2\be_2\in C$, so it follows that $C_\Qbb(\be_1,\be_2,\al)\cap\Ofrak_K=C$. 
\end{proof}

\noindent
In this special case of a quadratic extension, the elements of $\SG(\be_1,\be_2,\al)$ will also satisfy certain nice properties.

\begin{lemma}\label{SG in real quadratic}
Let $K$ be a real quadratic number field with positive integral basis $\be_1,\be_2\in\Ofrak_K^+$, and suppose that $\al=a_2\be_2-a_1\be_1\in\Ofrak_K^+$, where $a_1,a_2\in\Nbb$ are nonzero. Furthermore, assume that $q,r\in\Nbb$ are such that $0\leqslant r\leqslant a_1$ and $y_1=a_1q+r$. If $y_2\in\Zbb$, then 
\begin{enumerate}[label=(\alph*)]
    \item\label{SG in real quadratic a} if $0<r<a_1$, we have $y_2\be_2-y_1\be_1\in\SG(\be_1,\be_2,\al)$ if and only if $y_2\geqslant(q+1)a_2$; and
    \item\label{SG in real quadratic b} if $r=0$, meaning $y_1=qa_1$, we have $y_2\be_2-y_1\be_1\in\SG(\be_1,\be_2,\al)$ if and only if $y_2\geqslant qa_2$.
\end{enumerate}
\end{lemma}
\begin{proof}
\ref{SG in real quadratic a}: First suppose that $y_2\geqslant (q+1)a_2$. Then 
    $$y_2\be_2-y_1\be_1=(q+1)\al+((q+1)a_1-y_1)\be_1+(y_2-(q+1)a_2)\be_2\in\SG(\be_1,\be_2,\al)$$
since both $(q+1)a_1-y_1=a_1-r\geqslant0$ and $y_2-(q+1)a_2\geqslant0$. Conversely, suppose that $y_2\be_2-y_1\be_1\in\SG(\be_1,\be_2,\al)$, so there are $z,z_1,z_2\in\Nbb$ for which 
    $$y_2\be_2-y_1\be_1=z\al+z_1\be_1+z_2\be_2=(z_1-za_1)\be_1+(z_2+za_2)\be_2.$$
Hence $za_1-z_1=y_1$ and $z_2+za_2=y_2$, so $za_1-z_1=a_1q+r$, and the fact that $r>0$ shows that
    $$za_1=a_1q+r+z_1>a_1q,$$
so $z\geqslant q+1$, and thus
    $$y_2=z_2+za_2\geqslant z_2+(q+1)a_2\geqslant (q+1)a_2.$$ 
\ref{SG in real quadratic b}: First suppose that $y_2\geqslant qa_2$. Then 
    $$y_2\be_2-qa_1\be_1=q\al+(y_2-qa_2)\be_2\in\SG(\be_1,\be_2,\al)$$
because $y_2-qa_2\geqslant0$. Conversely, suppose that $y_2\be_2-qa_1\be_1\in\SG(\be_1,\be_2,\al)$, so there are are $z,z_1,z_2\in\Nbb$ for which 
    $$y_2\be_2-qa_1\be_1=z\al+z_1\be_1+z_2\be_2=(z_1-za_1)\be_1+(z_2+za_2)\be_2.$$
Hence $za_1-z_1=qa_1$ and $z_2+za_2=y_2$, so $za_1=qa_1+z_1\geqslant qa_1$, and thus $z\geqslant q$, so
    \[y_2=z_2+za_2\geqslant z_2+qa_2\geqslant qa_2.\qedhere\]
\end{proof}

\noindent
See figures \ref{picture for abb_2-ab_1} and \ref{picture for (m+1)b_2-mb_1} for a geometric interpretation of Lemmas \ref{C_Q in real quadratic} and \ref{SG in real quadratic}. Using Lemmas \ref{C_Q in real quadratic} and \ref{SG in real quadratic}, we can explicitly calculate what the Frobenius semigroup looks like for certain elements of $\Ofrak_K^+$. Specifically, let $\be_1,\be_2\in\Ofrak_K^+$ be a positive integral basis for $\Ofrak_K$, and suppose that $\al\in\Ofrak_K^+$ is in the form
    $$\al=ab\be_2-a\be_1,$$
where $a,b\in\Nbb$ are nonzero natural numbers. Then we claim that 
    $$\Frob(\be_1,\be_2,\al)=(a-1)\be_1+C_\Qbb(\be_1,\be_2,\al)\cap\Ofrak_K.$$
To do this, we first show the following, which is a stronger version of Lemma \ref{C_Q in real quadratic}.

\begin{lemma}\label{C_Q for abb_1-ab_2}
Let $K$ be a real number field and $\be_1,\be_2\in\Ofrak_K^+$ be an integral basis for $\Ofrak_K$. If $a,b\in\Nbb$ are nonzero and $\al=ab\be_2-a\be_1\in\Ofrak_K^+$, then
\begin{equation}
    C_\Qbb(\be_1,\be_2,\al)\cap\Ofrak_K=\left\{\left.\frac{n}{a}\al+n_1\be_1+n_2\be_2 \right|  n,n_1,n_2\in\Nbb\right\}=\SG\!\left(\be_1,\be_2,\frac{\al}{a}\right).\label{simple frob C_Q(b1,b2,a)}
\end{equation}
\end{lemma}
\begin{proof}
The inclusion 
    $$C_\Qbb(\be_1,\be_2,\al)\supseteq\SG\!\left(\be_1,\be_2,\frac{\al}{a}\right)$$
is obvious because $a$ divides the coefficients of both $\be_1$ and $\be_2$ in the expansion of $\al$ as a linear combination of them. Now suppose that $x,x_1,x_2\in\Qbb_{\geqslant 0}$ are such that 
    $$x\al+x_1\be_1+x_2\be_2=(x_1-xa)\be_1+(x_2+xab)\be_2\in C_\Qbb(\be_1,\be_2,\al)\cap\Ofrak_K.$$
If $x=0$ then there is nothing to show, so suppose that $x\neq0$. Then $x_2+xab\in\Nbb$ because it must be an integer and all of $x,a,b$ are nonzero, and $x_1-xa\in\Zbb$, so 
    $$(x_2+xab)+b(x_1-xa)=x_2+bx_1\in\Nbb.$$ 
By the division algorithm, there exists $n,n_2\in\Nbb$ for which $x_2+xab=nb+n_2$ and $0\leqslant n_2<b$. Define $n_1\in\Zbb$ so that $x_1-xa=n_1-n$. If $x_1=x_2=0$ then $x$ will clearly be in the form we want it to be, so suppose that at least one of $x_1$ or $x_2$ is nonzero, and thus strictly positive. Then we have that
    $$0<x_2+bx_1=x_2+xab+b(x_1-xa)=nb+n_2+n_1b-nb=n_2+n_1b.$$
Because $0\leqslant n_2<b$, the above shows that
    $$n_1>-\frac{n_2}{b}>-1,$$
and thus $n_1\geqslant0$. Hence $n_1\in\Nbb$, so we have $x_2+xab=nb+n_2$ and $x_1-xa=n_1-n$, where $n,n_1,n_2\in\Nbb$. It follows that 
\begin{align*}
    x\al+x_1\be_1+x_2\be_2&=(x_1-xa)\be_1+(x_2+xab)\be_2\\
    &=(n_1-n)\be_1+(nb+n_2)\be_2\\
    &=\frac{n}{a}(ab\be_2-a\be_1)+n_1\be_1+n_2\be_2\\
    &=\frac{n}{a}\al+n_1\be_1+n_2\be_2, 
\end{align*}
so we have the inclusion in the other direction in equation (\ref{simple frob C_Q(b1,b2,a)}), and consequently equation (\ref{simple frob C_Q(b1,b2,a)}) is true.
\end{proof}

\noindent
Using the results of Lemmas \ref{C_Q in real quadratic}, \ref{SG in real quadratic}, and \ref{C_Q for abb_1-ab_2}, we can calculate the set Frobenius semigroup $\Frob(\be_1,\be_2,\al)$ in certain cases.

\begin{proposition}\label{Frob for abb1-ab2 proposition}
Let $K$ be a real number field and $\be_1,\be_2\in\Ofrak_K^+$ be an integral basis. If $a,b\in\Nbb$ are nonzero and $\al=ab\be_2-a\be_1\in\Ofrak_K^+$, then
    $$\Frob(\be_1,\be_2,\al)=(a-1)\be_1+C_\Qbb(\be_1,\be_2,\al)\cap\Ofrak_K.$$
\end{proposition}
\begin{proof}
By Lemma \ref{condition for maximality}, in order to show that $\Mfrak(\be_1,\be_2,\al)=\{(a-1)\be_1\}$, it will suffice to show that $(a-1)\be_1\in\Frob(\be_1,\be_2,\al)$, and if $w\in\Frob(\be_1,\be_2,\al)$, then $w\preccurlyeq(a-1)\be_1$. Note that $n_1\be_1+n_2\be_2\in\SG(\be_1,\be_2,\al)$ whenever $n_1,n_2\in\Nbb$, so, in view of lemma \ref{C_Q for abb_1-ab_2}, we need only show that 
    $$(a-1)\be_1+\frac{n}{a}\al\in\SG(\be_1,\be_2,\al)$$
for all $n\in\Nbb$ in order to show that $(a-1)\be_1+C_\Qbb(\be_1,\be_2,\al)\cap\Ofrak_K\subseteq\SG(\be_1,\be_2,\al)$. If $0\leqslant n<a$, then 
    $$(a-1)\be_1+\frac{n}{a}\al=(a-1)\be_1+nb\be_2-n\be_1=nb\be_2+(a-1-n)\be_1\in\SG(\be_1,\be_2)\subseteq \SG(\be_1,\be_2,\al)$$
because $n\leqslant a-1$. If $n\geqslant a$, then we can perform the division algorithm to get $q,r\in\Nbb$ for which $0\leqslant r<a$ and $n=aq+r$, so
    $$(a-1)\be_1+\frac{n}{a}\al=(a-1)\be_1+\frac{r}{a}\al+q\al\in\SG(\be_1,\be_2,\al)$$
by the above. Hence $(a-1)\be_1\in\Frob(\be_1,\be_2,\al)$. \\
\\
We now claim that if $w\in\Frob(\be_1,\be_2,\al)$, then $w\preccurlyeq (a-1)\be_1$. In order to do this, we make use of the isomorphism $\vp:\Ofrak_K\to\Zbb^2$ associated to the basis $\be_1,\be_2$ for $\Ofrak_K$.

\begin{figure}[H]
    \begin{tikzpicture}[scale = 0.5]
    
     \draw[thin, color=gray, fill = gray, opacity=0.3] (-6,18.5) -- (-6,18) -- (-3,18) -- (-3,9) -- (0,9) -- (0,0) -- (5.5,0) --(5.5,18.5) -- (-6,18.5) ;
    
    
        
        \node[circle,inner sep=1, fill=black] at (0,0){};
        \draw (-9,0)--(9,0);
            \node at (9.8, 0) {$\be_1$};

        \draw (0,-1)--(0,19);
            \node at (0,19.5) {$\be_2$};

        \draw[thick] (0,0)--(-6,18);
    

        \foreach \i in {1,2,...,6}
            \node[circle, inner sep=1, fill=black] at (-\i,3*\i) {};

        \node at (-2,3) {$\frac{1}{a}\al$};
            \foreach \j in {4,5} 
                \node[circle, inner sep=1, fill=black] at (-1,\j) {};
        
        \node[rotate=-70] at (-2.5,4.6){$\dots$};
        
        \node at (-3,6) {$\frac{n}{a}\al$};
            \node[circle, inner sep=1, fill=black] at (0,6) {};

            \foreach \i in {1,2} 
                \node[circle, inner sep=1, fill=black] at (-\i,8) {};
            
            \node[circle, inner sep=1, fill=black] at (-1,6) {}; 
            
            \draw [decorate,decoration={brace,amplitude=5pt,mirror,raise=2pt},yshift=0pt] (0,6)--(-2,6) node [black,midway,yshift=12, xshift=-0.4] {\footnotesize$n+1$};
            
            \node[rotate=-70] at (-2.5-1,4.6+3){$\dots$};

        \node at (-4,9) {$\al$}; 
            \foreach \i in {0,...,2}
                \node[circle, inner sep=1, fill=black] at (-\i,9) {};
        
            \foreach \i in {1,2,3}
                \foreach \j in {9,10,...,18}
                    \node[circle, inner sep=1, fill=black] at (-\i,\j) {};
            
            \foreach \j in {12,...,18}
                \node[circle, inner sep=1, fill=black] at (-4,\j) {};
            
            \foreach \j in {15,...,18}
                \node[circle, inner sep=1, fill=black] at (-5,\j) {};
        
        \node at (-5.3,12) {$\frac{a+1}{a}\al$};
            \foreach \i in {0,...,3}
                \node[circle, inner sep=1, fill=black] at (-\i,12) {};
        
            \foreach \i in {0,...,4}
                \node[circle, inner sep=1, fill=black] at (-\i,15) {};
        
            \node[rotate=-70] at (-2.5-4,4.6+12){$\dots$};
            
            \node[rotate=-70] at (-2.5-3,4.6+9){$\dots$};
            \node[rotate=-70] at (-2.5-3*1.16,4.6+9*1.16){$\dots$};
        
        \node at (-6.8,18) {$2\al$};
            \foreach \i in {0,...,5}
                \node[circle, inner sep=1, fill=black] at (-\i,18) {};

    \foreach \j in {1,...,18}
        \node[circle, inner sep=1, fill=black] at (0,\j) {};
        
        \draw [decorate,decoration={brace,amplitude=5pt,mirror,raise=2pt},yshift=0pt] (0,0)--(0,6) node [black,midway,yshift=0, xshift=20] {\footnotesize$nb+1$};
        
        \draw [decorate,decoration={brace,amplitude=5pt,mirror,raise=2pt},yshift=0pt] (5.4,0)--(5.4,9) node [black,midway,yshift=0, xshift=20] {\footnotesize$ab+1$};

        \foreach \i in {1,2} 
            \foreach \j in {0,1,2,4,5,6,...,18}
                \node[circle, inner sep=1, fill=black] at (\i,\j) {};
         
        \foreach \i in {3,4,5} 
            \foreach \j in {0,1,...,18}
                \node[circle, inner sep=1, fill=black] at (\i,\j) {};
        
        
    \end{tikzpicture}
    \caption{An illustration of the sets $\vp(\SG(\be_1,\be_2,\al))$ and $\vp(C_\Qbb(\be_1,\be_2,\al)\cap\Ofrak_K)$ in $\Zbb^2$. Black points represent elements of $\vp(C_\Qbb(\be_1,\be_2,\al)\cap\Ofrak_K)$, and the shaded region represents points of $\vp(\SG(\be_1,\be_2,\al))$. The thick black line represents the set $\vp(C_\Qbb(\al)\cap\Ofrak_K)$, which along with the positive $\be_1$ axis, marks the boundary of the set $\vp(C_\Qbb(\be_1,\be_2,\al)\cap\Ofrak_K)$.}\label{picture for abb_2-ab_1}
\end{figure}

\noindent
Note that the rational multiples of $\vp(\al)$ in $\vp(C_\Qbb(\be_1,\be_2,\al)\cap\Ofrak_K)$ all lie along the line $y=-bx$ in $\Zbb^2$. The region obtained by shifting the cone $\vp(C_\Qbb(\be_1,\be_2,\al)\cap\Ofrak_K)$ to the right by $(a-1)\vp(\be_1)$ will be bounded by the lines $y=0$ and $y=-bx+b(a-1)$. It follows that if $x_1\in\Zbb$, $x_2\in\Nbb$, and $x_1\be_1+x_2\be_2\in\Frob(\be_1,\be_2,\al)$, then $x_1\be_1+x_2\be_2\preccurlyeq(a-1)\be_1$ if and only if $x_2\geqslant -bx_1+b(a-1)$.\\
\\
In order to show that every element in $\Frob(\be_1,\be_2,\al)$ precedes $(a-1)\be_1$, we leverage Lemma \ref{SG in real quadratic} and use the observations at the end of the previous paragraph. Suppose that $x_1,x_2\in\Nbb$ and $w=x_1\be_1+x_2\be_2\in\Frob(\be_1,\be_2,\al)$, so we claim that $w\preccurlyeq(a-1)\be_1$. Then
    $$\frac{x_1+1}{a}\al+x_1\be_1+x_2\be_2=((x_1+1)b+x_2)\be_2-\be_1\in\SG(\be_1,\be_2,\al),$$
so Lemma \ref{SG in real quadratic} shows that we must have
    $$(x_1+1)b+x_2\geqslant ab.$$
This is equivalent to  
    $$x_2\geqslant -bx_1+b(a-1),$$
meaning $x_1\be_1+x_2\be_2\preccurlyeq(a-1)\be_1$.\\
\\
Now suppose that $x_1,x_2\in\Nbb$ are nonzero and $w=x_2\be_2-x_1\be_1\in\Frob(\be_1,\be_2,\al)$ (we can ignore the case where $x_2<0$ because no point in that form will be in $\SG(\be_1,\be_2,\al)$). By the division algorithm, there are $q,r\in\Nbb$ for which $x_1=qa+r$ and $0\leqslant r<a$. Then since $(a-r+1)\al/a\in C_\Qbb(\be_1,\be_2,\al)\cap\Ofrak_K$ by lemma \ref{C_Q for abb_1-ab_2}, we know that
\begin{align*}
    \frac{a-r+1}{a}\al+x_2\be_2-x_1\be_1&=(a-r+1)b\be_2-(a-r+1)\be_1+x_2\be_2-x_1\be_1\\
    &=((a-r+1)b+x_2)\be_2-((q+1)a+1)\be_1\in\SG(\be_1,\be_2,\al),
\end{align*}
and Lemma \ref{SG in real quadratic} shows that in order for this to be true, we must have
    $$(a-r+1)b+x_2\geqslant (q+2)ab,$$
which implies that 
    $$x_2\geqslant qab+2ab-ab+rb-b=bx_1+b(a-1).$$
Hence $x_2\be_2-x_1\be_1\preccurlyeq (a-1)\be_1$, so Lemma \ref{condition for maximality} shows that $(a-1)\be_1$ is the only maximal element of $\Frob(\be_1,\be_2,\al)$, and thus
    \[\Frob(\be_1,\be_2,\al)=(a-1)\be_1+C_\Qbb(\be_1,\be_2,\al)\cap\Ofrak_K.\qedhere\]
\end{proof}

\noindent
If we take $a=1$, then Lemma \ref{C_Q for abb_1-ab_2} and Proposition \ref{Frob for abb1-ab2 proposition} show that
    $$\Frob(\be_1,\be_2,b\be_2-\be_1)=C_\Qbb(\be_1,\be_2,b\be_2-\be_1)\cap\Ofrak_K=\SG(\be_1,\be_2,b\be_2-\be_1).$$
This gives us an example of a real number field $K$ and a collection of elements $\al_1,\dots,\al_n\in\Ofrak_K^+$ that is not a basis for $\Ofrak_K$ for which $\Frob(\al_1,\dots,\al_n)=\SG(\al_1,\dots,\al_n)$. Furthermore, if we take $a>1$ then we get a collection of elements in the form $\al=ab\be_2-a\be_1\in\Ofrak_K^+$ for which $\Frob(\be_1,\be_2,\al)$ never contains $0$, and thus
    $$\Frob(\be_1,\be_2)\not\subseteq\Frob(\be_1,\be_2,\al).$$
The existence of such elements $\al\in\Ofrak_K^+$ was promised at the end of section \ref{The Frobenius Problem Section}.

\section{The Number of Maximal Elements}\label{The Number of Maximal Elements Section}
From the results of sections \ref{The Structure of the Frobenius Set Section} and \ref{An Example of a Simple Frobenius Set For Quadratic Extensions Section}, it may be tempting to conclude that for a real number field $K$ and a fixed value of $n$, the number $\#\Mfrak(\al_1,\dots,\al_n)$ will be bounded above as the nonzero elements $\al_1,\dots,\al_n\in\Ofrak_K^+$ range over spanning sets of $\Ofrak_K$. However, we now show that this is not the case, and in fact, when $K$ is a quadratic extension we can make $\#\Mfrak(\al_1,\al_2,\al_3)$ arbitrarily large. In particular, we have the following:

\begin{proposition}\label{Frob set in quadratic extension with m maximal elements}
Let $K$ be a real quadratic number field and fix a positive integral basis $\be_1,\be_2\in\Ofrak_K^+$ for $\Ofrak_K$ with $\be_1<\be_2$. If $m>1$ is an integer and $\al=(m+1)\be_2-m\be_1$, then
    $$\Mfrak(\be_1,\be_2,\al)=\{(m-i)\be_1+(i-1)\be_2 \mid i=1,\dots,m\},$$
and thus
    $$\Frob(\be_1,\be_2,\al)=\bigcup_{i=1}^{m}\big((m-i)\be_1+(i-1)\be_2+C_\Qbb(\be_1,\be_2,\al)\cap\Ofrak_K\big).$$
\end{proposition}
\begin{proof}
Note that $\al\in\Ofrak_K^+$ because $\be_1<\be_2$. For each $i=1,\dots,m$, let 
    $$\mu_i=(m-i)\be_1+(i-1)\be_2,$$
so we claim $\Mfrak(\be_1,\be_2,\al)=\{\mu_1,\dots,\mu_m\}$. \\
\\
To do this, we first show that each $\mu_i\in\Frob(\be_1,\be_2,\al)$, and then apply Lemma \ref{condition for maximality}. Recall, by Lemma \ref{C_Q in real quadratic}, that the set $C_\Qbb(\be_1,\be_2,\al)\cap\Ofrak_K$ is given by
    $$C_\Qbb(\be_1,\be_2,\al)\cap\Ofrak_K=\left\{x_1\be_1+x_2\be_2 \left|\, x_1,x_2\in\Zbb, \ x_2\geqslant 0 \,\text{ and } 
    x_2\geqslant-\frac{m+1}{m}x_1\right.\right\}.$$
If $x_1,x_2\in\Nbb$ then it is clear that $\mu_i+x_1\be_1+x_2\be_2\in\SG(\be_1,\be_2,\al)$ because $\mu_i,x_1\be_1+x_2\be_2\in\SG(\be_1,\be_2,\al)$, so suppose that $x_1,x_2\in\Nbb$ and $x_2\be_2-x_1\be_1\in C_\Qbb(\be_1,\be_2,\al)\cap\Ofrak_K$. First assume that $x_1\leqslant m$. Then
    $$\mu_i+x_2\be_2-x_1\be_1=(m-i-x_1)\be_1+(i-1+x_2)\be_2,$$
and if $x_1\leqslant m-i$ then the above is clearly in $\SG(\be_1,\be_2)$, so suppose that $x_1>m-i$. Then the fact that $x_1\leqslant m$ implies that $x_1\leqslant 2m-i$, and we can thus write 
\begin{align*}
    \mu_i+x_2\be_2-x_1\be_1&=((m+1)\be_2-m\be_1)+((2m-i-x_1)\be_1+(x_2+i-2-m)\be_2)\\
    &=\al+(2m-i-x_1)\be_1+(x_2+i-2-m)\be_2.
\end{align*}
We know that $2m-i-x_1\geqslant 0$, and the fact that $x_1>m-i$ implies that $x_1\geqslant m-i+1$. The fact that $x_2\be_2-x_1\be_1\in C_\Qbb(\be_1,\be_2,\al)\cap\Ofrak_K$, combined with Lemma \ref{C_Q in real quadratic}, then shows that 
    $$x_2\geqslant \frac{m+1}{m}x_1\geqslant\frac{m+1}{m}(m-i+1)=m+2+\frac{1}{m}-i-\frac{i}{m}.$$
Hence
    $$x_2-m-2+i\geqslant \frac{1}{m}-\frac{i}{m}\geqslant \frac{1}{m}-1>-1,$$
so $x_2-m-2+i\geqslant0$, and thus
\begin{equation}
    \mu_i+x_2\be_2-x_1\be_1=\al +(2m-i-x_1)\be_1+(x_2+i-2-m)\be_2\in\SG(\al,\be_1,\be_2).\label{mui+x2b2+x1b1 for x1<m}
\end{equation}
Now suppose that $x_1>m$, $x_2\be_2-x_1\be_1\in C_\Qbb(\be_1,\be_2,\al)\cap\Ofrak_K$, and write $x_1=qm+r_1$ for $q,r_1\in\Nbb$ and $0\leqslant r_1<m$. If $x_2\geqslant (q+1)(m+1)$ and $r_1>0$ then part \ref{SG in real quadratic a} of Lemma \ref{SG in real quadratic} shows that $x_2\be_2-x_2\be_1\in\SG(\be_1,\be_2,\al)$, and if $x_2\geqslant (q+1)(m+1)$ and $r_1=0$ then part \ref{SG in real quadratic b} of Lemma \ref{SG in real quadratic} shows that $x_2\be_2-x_1\be_1\in\SG(\be_1,\be_2,\al)$, since $x_2\geqslant (q+1)(m+1)>q(m+1)$. We thus suppose that $x_2<(q+1)(m+1)=q(m+1)+m+1$. The condition that $x_2\be_2-x_1\be_1\in C_\Qbb(\be_1,\be_2,\al)\cap\Ofrak_K$, combined with Lemma \ref{C_Q in real quadratic}, then gives 
    $$x_2\geqslant \frac{m+1}{m}x_1=\frac{m+1}{m}(qm+r_1)=q(m+1)+r_1\frac{m+1}{m}.$$
Hence $q(m+1)\leqslant x_2<(q+1)(m+1)$, so we can write $x_2=q(m+1)+r_2$ for some $0\leqslant r_2<m+1$, meaning 
    $$x_2\be_2-x_1\be_1=q(m+1)\be_2+r_2\be_2-qm\be_1-r_1\be_1=q\al+r_2\be_2-r_1\be_1.$$
The fact that $0\leqslant r_1<m$, combined with what we showed in equation (\ref{mui+x2b2+x1b1 for x1<m}), then gives that
    $$\mu_i+x_2\be_2-x_1\be_1=q\al+(\mu_i+r_2\be_2-r_1\be_1)\in\SG(\be_1,\be_2,\al).$$
Thus $\mu_1,\dots,\mu_m\in\Frob(\be_1,\be_2,\al)$.\\
\\
We now apply Lemma \ref{condition for maximality} in order to show that $\Mfrak(\be_1,\be_2,\al)=\{\mu_1,\dots,\mu_m\}$, so we must first show that $\mu_i\gprecneq\mu_j$ for any distinct $i,j\in\{1,\dots,m\}$. Let $\vp:\Ofrak_K\to\Zbb^2$ be the $\Zbb$-module isomorphism associated to the basis $\be_1,\be_2$ for $\Ofrak_K$, and let $\pi_i:\Ofrak_K\to\Zbb$, $i=1,2$, be the projection of $\Ofrak_K=\Zbb\be_1\oplus\Zbb\be_2$ onto $\Zbb\be_i$. Note that if $\ga_1,\ga_2\in \Frob(\be_1,\be_2,\al)$ and $\pi_2(\ga_1)<\pi_2(\ga_2)$, then $\ga_1\gprecneq\ga_2$ because $\pi_2(C_\Qbb(\be_1,\be_2,\al)\cap\Ofrak_K)\subseteq\Nbb$ by Lemma \ref{C_Q in real quadratic}. Hence if $i<j$ then $\mu_i\gprecneq\mu_j$ because
    $$\pi_2(\mu_i)=i-1<j-1=\pi_2(\mu_j).$$
Now, note that the slope of the line in $\Zbb^2$ connecting $\vp(\mu_i)$ to $\vp(\mu_j)$ is $-1$, while, by Lemma 13, the slope of the line extending out of $\vp(\mu_i)$ that determines one edge of the boundary of $\vp(\mu_i)+\vp(C_\Qbb(\be_1,\be_2,\al)\cap\Ofrak_K)$ is
    $$-\frac{m+1}{m}<-1.$$
Hence $\vp(\mu_j)$ cannot be contained in $\vp(\mu_i)+\vp(C_\Qbb(\be_1,\be_2,\al)\cap\Ofrak_K)$, so $\mu_j$ cannot be contained in $\mu_i+C_\Qbb(\be_1,\be_2,\al)\cap\Ofrak_K$, and thus it is not possible that $\mu_j\preccurlyeq\mu_i$.

\begin{figure}[H]
    \begin{tikzpicture}[scale = 0.5]

        \draw[thin, color=gray, fill = gray, opacity=0.3] (-8,10.2)--(-8,10)--(-4,10)--(-4,5)--(0,5)--(0,0)--(8.2,0)--(8.2,10.2)--(-8,10.2);

        \node[circle,inner sep=1, fill=black] at (0,0){};

        \draw[thin] (-9,0)--(9,0);
            \node at (9.5,0) {$\be_1$};

        \draw[thin] (0,-0.5)--(0,11);
            \node at (0.1,11.4) {$\be_2$};

        \draw[scale=1, domain=0:-8, smooth, variable=\x, black, thick] plot ({\x}, {(-5/4)*\x});
        \draw[thick] (0,0)--(9,0);

        \node[circle, inner sep=1, fill=black] at (-4,5) {};
            \node at (-4.5,5) {$\alpha$};

        \node[circle, inner sep=1, fill=black] at (-8,10) {};
            \node at (-8.6,10.1) {$2\al$};

        \foreach \x in {0,...,8}
            \foreach \y in {0,...,10}
                \node[circle, inner sep=1, fill=black] at (\x,\y) {};
        
        \foreach \x in {0,...,-4}
            \foreach \y in {-\x,...,9}
                \node[circle, inner sep=1 pt, fill=black] at (\x,\y+1) {};

        \foreach \x in {-5,...,-8}
            \foreach \y in {-\x,...,8}
                \node[circle, inner sep=1 pt, fill=black] at (\x,\y+2) {};


        \foreach \i in {1,...,4} 
            \draw[scale=1, domain={4-\i}:{-5-(0.2)*(\i-1))}, smooth, variable=\x, black, very thick] plot ({\x}, {(-5/4)*(\x-(4-\i))-1+\i});

        \foreach \i in {2,...,4} 
            \node at (4-\i-0.5,\i-1) {\footnotesize$\mu_{\i}$};
        \node at (3.2,-0.4) {\footnotesize$\mu_1$};
        
        \foreach \i in {1,...,4}
            \node[circle, inner sep=1, fill=black] at (4-\i,\i-1) {};

        \foreach \i in {1,...,4} 
            \draw[scale=1, domain={4-\i}:{8}, smooth, variable=\x, black, very thick] plot ({\x}, {\i-1});

    \end{tikzpicture}
    \caption{An illustration of the sets $\vp(\SG(\be_1,\be_2,\al)),\vp(C_\Qbb(\be_1,\be_2,\al)\cap\Ofrak_K)$, and $\vp(\mu_i)+\vp(C_\Qbb(\be_1,\be_2,\al)\cap\Ofrak_K)$ in $\Zbb^2$ for the case $m=4$. The black points represent elements of $\vp(C_\Qbb(\be_1,\be_2,\al)\cap\Ofrak_K)$, the shaded region represents elements of $\vp(\SG(\be_1,\be_2,\al))$, and the thick lines coming out of each $\mu_i$ represent part of the boundary of the set $\vp(\mu_i)+\vp(C_\Qbb(\be_1,\be_2,\al)\cap\Ofrak_K)$. }\label{picture for (m+1)b_2-mb_1}
\end{figure}

\noindent 
We now show that if $w\in\Frob(\be_1,\be_2,\al)$, then $w\preccurlyeq\mu_i$ for some $i=1,\dots,m$. We first deal with the elements of $\Frob(\be_1,\be_2,\al)$ with a non-negative $\be_1$ coefficient, meaning we claim that if $x_1,x_2\in\Nbb$ and $w=x_1\be_1+x_2\be_2\in\Frob(\be_1,\be_2,\al)$, then $w\preccurlyeq\mu_i$ for some $i=1,\dots,m$. Note that in $\Zbb^2$, the points $\vp(\mu_1),\dots,\vp(\mu_m)$ all lie along the line $y=m-1-x$, and furthermore, the points $\vp(\mu_1),\dots,\vp(\mu_m)$ consist of all integral points along this line with non-negative $x$ and $y$ coordinates. Thus if $x_1,x_2\in\Nbb$ and $w=x_1\be_1+x_2\be_2$, then $w\preccurlyeq\mu_i$ for some $i$ if and only if $x_2\geqslant m-1-x_1$. If $x_1\geqslant m-1$ then $m-1-x_1\leqslant0\leqslant x_2$ because $x_2\in\Nbb$, so we know that $w\preccurlyeq\mu_i$ for some $i$. Now suppose that $w=x_1\be_1+x_2\be_2\in\Frob(\be_1,\be_2,\al)$ and $x_1<m-1$. We claim that $w=x_1\be_1+x_2\be_2\preccurlyeq\mu_i$ for some $i=1,\dots,m$. By Lemma \ref{C_Q in real quadratic}, we know that $(2+x_1)\be_2-(1+x_1)\be_1\in C_\Qbb(\be_1,\be_2,\al)\cap\Ofrak_K$ because 
    $$\frac{m+1}{m}(1+x_1)=x_1+1+\frac{x_1+1}{m}<x_1+2.$$
Now, 
    $$w+(2+x_1)\be_2-(1+x_1)\be_1=(2+x_1+x_2)\be_2-\be_1,$$
so the fact that $w\in\Frob(\be_1,\be_2,\al)$ shows that $(2+x_1+x_2)\be_2-\be_1\in\SG(\be_1,\be_2,\al)$, and Lemma \ref{SG in real quadratic} then shows that 
    $$2+x_1+x_2\geqslant m+1.$$ 
Then
    $$x_2\geqslant m-1-x_1,$$
so $w\preccurlyeq\mu_i$ for some $i$. \\
\\
We now deal with the elements of $\Frob(\be_1,\be_2,\al)$ with a negative $\be_1$ coefficient, meaning we claim that if $x_1,x_2\in\Nbb$ and $w=x_2\be_2-x_1\be_1\in\Frob(\be_1,\be_2,\al)$, then $w\preccurlyeq\mu_m$. As before, note that the line in $\Zbb^2$ determining part of the boundary of the cone $\vp(\mu_m)+\vp(C_\Qbb(\be_1,\be_2,\al)\cap\Ofrak_K)$ is $y=-\frac{m+1}{m}x+m-1$, so we have $w\preccurlyeq\mu_m$ if and only if $x_2\geqslant \frac{m+1}{m}x_1+m-1$. Let $x_1=qm+r$, where $q,r\in\Nbb$ and $0\leqslant r< m$. Then if $r\geqslant 1$, 
    $$(m-r+2)\be_2-(m-r+1)\be_1\in C_\Qbb(\be_1,\be_2,\al)\cap\Ofrak_K$$
by Lemma \ref{C_Q in real quadratic}, since
    $$\frac{m+1}{m}(m-r+1)=m+1+\frac{m+1}{m}(1-r)\leqslant m+1+1-r=m-r+2.$$
Adding $w$ to this point yields
\begin{align*}
    w+(m-r+2)\be_2-(m-r+1)\be_1&=(m-r+2+x_2)\be_2-(m-r+1+x_1)\be_1\\
    &=(m-r+2+x_2)\be_2-((q+1)m+1)\be_1,
\end{align*}
and the fact that $w\in\Frob(\be_1,\be_2,\al)$ then shows that this element is in $\SG(\be_1,\be_2,\al)$. Lemma \ref{SG in real quadratic} then implies that we must have 
    $$m-r+2+x_2\geqslant (q+2)(m+1).$$
It follows that
\begin{align*}
    x_2&\geqslant q(m+1)+m+r=\frac{x_1-r}{m}(m+1)+m+r=\frac{m+1}{m}x_1+m-\frac{r}{m}> \frac{m+1}{m}x_1+m-1,
\end{align*}
so $w=x_2\be_2-x_1\be_1\preccurlyeq\mu_m$.\\
\\
If $r=0$ then $x_1=qm$, the point $2\be_2-\be_1\in C_\Qbb(\be_1,\be_2,\al)\cap\Ofrak_K$ because $2\geqslant\frac{m+1}{m}$, and we have
    $$w+2\be_2-\be_1=x_2\be_2-qm\be_1+2\be_2-\be_1=(2+x_2)\be_2-(qm+1)\be_1\in\SG(\be_1,\be_2,\al).$$
By Lemma \ref{SG in real quadratic}, this means that $2+x_2\geqslant (q+1)(m+1),$ so
    $$x_2\geqslant q(m+1)+m-1=\frac{m+1}{m}x_1+m-1,$$
and thus $w\preccurlyeq\mu_m$. It follows that conditions \ref{maxcond1} and \ref{maxcond2} of Lemma \ref{condition for maximality} are satisfied, so $\Mfrak(\be_1,\be_2,\al)=\{\mu_1,\dots,\mu_m\}$, and thence the set $\Frob(\be_1,\be_2,\al)$ is given by 
    \[\Frob(\be_1,\be_2,\al)=\bigcup_{i=1}^{m}\big((m-i)\be_1+(i-1)\be_2+C_\Qbb(\be_1,\be_2,\al)\cap\Ofrak_K\big).\qedhere\]
\end{proof}

\noindent
Note that in the above proof, the elements $\mu_2,\dots,\mu_{m-1}$ are sort of ``fringe" maximal elements. That is, with the exception of $\mu_2,\dots,\mu_{m-1}$, every element of $\Frob(\be_1,\be_2,\al)$ will either precede $\mu_1$ or $\mu_m$.\\
\\
While the above proposition shows that in the case of quadratic extensions, we can make $\#\Mfrak(\al_1,\al_2,\al_3)$ arbitrarily large, we strongly believe that something similar will be the case for arbitrary real number fields.

\begin{conjecture}\label{conjecture about size of set of maximal elements}
Let $K$ be a real number field and $\be_1,\dots,\be_d\in\Ofrak_K^+$ be a basis for $\Ofrak_K$ as a $\Zbb$-module. Then for any positive integer $m\geqslant 1$, there is some $\al\in\Ofrak_K^+$ for which $\#\Mfrak(\be_1,\dots,\be_d,\al)\geqslant m$. 
\end{conjecture}

\noindent
Given a real number field $K$ and $\al_1,\dots,\al_n\in\Ofrak_K^+$ that generate $\Ofrak_K$ as a $\Zbb$-module, another interesting question to consider is how the elements of $\Mfrak(\al_1,\dots,\al_n)$ are related to each other. Lemma \ref{maximal elements are pairwise linearly independent} shows that the elements of $\Mfrak(\al_1,\dots,\al_n)$ are all pairwise linearly independent, but the calculations done in sections \ref{An Example of a Simple Frobenius Set For Quadratic Extensions Section} and \ref{The Number of Maximal Elements Section} suggest that there is likely much more to be seen. For example, all of the maximal elements in the set $\Frob(\be_1,\be_2,\al)$ in Proposition \ref{Frob set in quadratic extension with m maximal elements} are collinear when considered as elements in $\Zbb^2$ under the isomorphism $\Ofrak_K\to\Zbb^2$ associated to the basis $\be_1,\be_2$. Suppose that $K$ is an arbitrary real number field of degree $d$ with integral basis $\be_1,\dots,\be_d\in\Ofrak_K^+$, and we have elements $\al_1,\dots,\al_n\in\Ofrak_K^+$. A natural question to ask is then whether there is some nice geometric description of the elements of $\Mfrak(\be_1,\dots,\be_d,\al_1,\dots,\al_n)$ when they are considered as elements in $\Zbb^d$ under the isomorphism $\Ofrak_K\to\Zbb^d$ associated to the basis $\be_1,\dots,\be_d$ for $\Ofrak_K$?\\
\\
It is also important to note that all of the explicit calculations of Frobenius semigroups that we have so far only consider cases where the collection of elements $\al_1,\dots,\al_n\in\Ofrak_K^+$ contain a basis for $\Ofrak_K$. If $\al_1,\dots,\al_n\in\Ofrak_K^+$ span $\Ofrak_K$ as a $\Zbb$-module, then there will certainly be some collection $\al_{i_1},\dots,\al_{i_d}$, where $d=[K:\Qbb]$, that is linearly independent. However, in this case the elements $\al_{i_1},\dots,\al_{i_d}$ need not span $\Ofrak_K$. Focusing on the quadratic case, let us assume that $\be_1,\be_2\in\Ofrak_K^+$ are linearly independent, and $a_1,a_2\in\Qbb_{\geqslant0}$ are such that $\be_1,\be_2,a_2\be_2-a_1\be_1\in\Ofrak_K^+$ span $\Ofrak_K$ as a $\Zbb$-module. Then the nice conclusions of Lemmas \ref{C_Q in real quadratic} and \ref{SG in real quadratic} may not hold, since their proofs rely on the assumption that $\be_1,\be_2$ is an integral basis for $\Ofrak_K$. This shows that calculations of the Frobenius semigroup in cases where the elements do not contain a basis could potentially be much more complicated. \\
\\
Another thing to consider, is given some real number field $K$ of degree $d$ and $\al_1,\dots,\al_n\in\Ofrak_K^+$ that generate $\Ofrak_K$ as a $\Zbb$-module, is there some nice bound on $\#\Mfrak(\al_1,\dots,\al_n)$ in terms of $n,d$, and $\al_1,\dots,\al_n$? Proposition \ref{Frob set in quadratic extension with m maximal elements} shows that this bound cannot depend on only $n$ and $d$, but it does not rule out the possibility of the bound also depending on $\al_1,\dots,\al_n$. Specifically, there may be some bounds on $\#\Mfrak(\al_1,\dots,\al_n)$ that depend on more number theoretic properties of $\Ofrak_K$, perhaps similar in nature to some of the bounds appearing in \cite{integerknaps}, \cite{Polyhedra}, or \cite{TotallyReal}.

\printbibliography

\end{document}

%% file: setup.tex
\usepackage{amssymb}
\usepackage[font={small, it}]{caption}
 \usepackage{dsfont}
\usepackage{amsmath}
\usepackage{floatrow}
\usepackage{halloweenmath}
\usepackage{times}

\usepackage{stmaryrd}
\usepackage{amsthm}
\usepackage{authoraftertitle}
\usepackage{xcolor}
\usepackage{mathrsfs}
\usepackage[hyphens]{url}
\usepackage[colorlinks = true,
            linkcolor = black,
            urlcolor  = blue,
            citecolor = black,
            anchorcolor = blue]{hyperref}
\usepackage[hyphenbreaks]{breakurl}
\usepackage[mathscr]{euscript}
\usepackage{wasysym}
\usepackage[sorting=ynt, maxnames=50]{biblatex}
\DeclareFieldFormat*{title}{#1}
\usepackage[letterpaper, portrait, margin=1in]{geometry}
\usepackage{graphicx}
\usepackage{tikz}
\usepackage{tikz-3dplot}
\usepackage{pgfplots}
\usetikzlibrary{decorations.pathmorphing, patterns}
\usepackage{lipsum}
\usepackage{float}
\usepackage{subcaption}
\usepackage[object=vectorian]{pgfornament}
\usepackage{mwe}
\usepackage{bigints}
\usepackage{titlesec}
\titleformat{\paragraph}
{\normalfont\normalsize\bfseries}{\theparagraph}{1em}{}
\titlespacing*{\paragraph}
{0pt}{3.25ex plus 1ex minus .2ex}{1.5ex plus .2ex}
\usepackage{mathtools}
\usepackage{pgfplots}
\pgfplotsset{compat=1.15}
\usepackage{lastpage}
\usepackage{enumitem}
\usepackage{tensor}
\usepackage{mathtools}
\usepackage{fancyhdr} 
\usepackage{setspace}
\usepackage{tikz}
\usetikzlibrary{hobby}

\usepackage{pst-node}

\usepackage{tikz-cd}

\usepackage{thmtools}
\usepackage{thm-restate}

\usepackage{cleveref}

\theoremstyle{definition}
\newtheorem{definition}{Definition}[]

\theoremstyle{plain}
\declaretheorem[name=Theorem]{theorem}

\newtheorem{lemma}[theorem]{Lemma}
\newtheorem{conjecture}{Conjecture}

\newtheorem{corollary}[theorem]{Corollary}

\newtheorem{proposition}[theorem]{Proposition}

    \newcommand{\al}{\alpha}
    \newcommand{\be}{\beta}
    \newcommand{\Nbb}{\mathbb{N}}
    \newcommand{\Ofrak}{\mathfrak{O}}
    \newcommand{\Si}{\Sigma}

    \newcommand{\si}{\sigma}

    \newcommand{\Mfrak}{\mathfrak{M}}

    \DeclareMathOperator{\inte}{int}

    \newcommand{\Fcal}{\mathcal{F}}

    \newcommand{\vp}{\varphi}

    \newcommand{\Zbb}{\mathbb{Z}}

    \newcommand{\Tcal}{\mathcal{T}}

    \newcommand{\Qbb}{\mathbb{Q}}
    \newcommand{\ga}{\gamma}

    \newcommand{\Rbb}{\mathbb{R}}

    \newcommand{\afrak}{\mathfrak{a}}
    
    \newcommand{\Cbb}{\mathbb{C}}

    \DeclareMathOperator{\SG}{SG}
    \DeclareMathOperator{\Frob}{Frob}

\newcommand{\goodemptyset}[1]{%
\begin{tikzpicture}[#1]%
\draw[line width=0.45] (0,0) circle (0.1);%
\draw[line cap=round, line width=0.45] (-0.07,-0.14)--(0.07,0.14);
\end{tikzpicture}%
}

\newcommand{\es}{\raisebox{-1.2pt}{\goodemptyset{}}}

\newcommand{\goodpreceq}[1]{%
\begin{tikzpicture}[#1]%
\node at (0,0) {$\preccurlyeq$};%
\draw[line cap=round, line width=0.45](-0.04,-0.15)--(0.04,0.15);
\end{tikzpicture}%
}

\newcommand{\gprecneq}{\raisebox{-5pt}{\goodpreceq{}}}

\makeatletter
\DeclareRobustCommand{\cev}[1]{%
  {\mathpalette\do@cev{#1}}%
}
\newcommand{\do@cev}[2]{%
  \vbox{\offinterlineskip
    \sbox\z@{$\m@th#1 x$}%
    \ialign{##\cr
      \hidewidth\reflectbox{$\m@th#1\vec{}\mkern4mu$}\hidewidth\cr
      \noalign{\kern-\ht\z@}
      $\m@th#1#2$\cr
    }%
  }%
}
\makeatother

\makeatletter
\DeclarePairedDelimiterX{\pmodx}[1]{(}{)}{{\operator@font mod}\mkern6mu#1}
\renewcommand{\pmod}{%
  \allowbreak
  \if@display\mkern18mu\else\mkern8mu\fi
  \pmodx
}
\makeatother

\DeclarePairedDelimiterX\braket[2]{\langle}{\rangle}{#1 \delimsize\vert #2}

\makeatletter
\newcommand{\colim@}[2]{%
  \vtop{\m@th\ialign{##\cr
    \hfil$#1\operator@font colim$\hfil\cr
    \noalign{\nointerlineskip\kern1.5\ex@}#2\cr
    \noalign{\nointerlineskip\kern-\ex@}\cr}}%
}
\newcommand{\colim}{%
  \mathop{\mathpalette\colim@{\rightarrowfill@\scriptscriptstyle}}\nmlimits@
}
\renewcommand{\varinjlim}{%
  \mathop{\mathpalette\varlim@{\rightarrowfill@\scriptscriptstyle}}\nmlimits@
}
\newcommand{\inlim}{%
  \mathop{\mathpalette\varlim@{\leftarrowfill@\scriptscriptstyle}}\nmlimits@
}

\providecommand{\keywords}[1]
{
  \small	
  {\textit{Keywords:}} #1
}